\documentclass[times,doublespace]{oupau}

\usepackage{times,amsfonts,amsmath,amstext,amsbsy,amssymb,amsopn,amsthm,upref,eucal,amscd}
\usepackage[T1]{fontenc}
\usepackage{color}
\usepackage{verbatim} 
\usepackage{graphicx} 
\usepackage{rotating} 

\usepackage[all,arc,knot]{xy}
\usepackage{mathtools}
\xyoption{arc}

\theoremstyle{remark}
\newtheorem{remark}[theorem]{Remark}

\renewcommand{\epsilon}{\varepsilon}
\renewcommand{\phi}{\varphi}

\newcommand{\field}[1]{\mathbb{#1}}

\newcommand{\R}{\field{R}}
\newcommand{\N}{\field{N}}
\newcommand{\C}{\field{C}}
\newcommand{\Z}{\field{Z}}

\newcommand{\cH}{{\mathcal H}}

\newcommand{\SLZ}{\operatorname{SL}(2,{\mathbb Z})}

\newcommand{\wmsa}{{\sigma_a}^{\hspace*{-1.6mm}M_3}}
\newcommand{\wmsb}{{\sigma_b}^{\hspace*{-1.6mm}M_3}}
\newcommand{\Ornicov}{CovM_4}
\newcommand{\Orni}{M_4}
\newcommand{\Ornitwo}{\tilde{M}_4}
\newcommand{\scs}{\scriptsize}

\newcommand{\lcmfs}{\textrm{lcm}}

\newlength{\halfbls}
\begin{document}

\title{Some examples of isotropic {\boldmath $SL(2,\mathbb{R})$}-invariant subbundles of the Hodge bundle}
\shorttitle{Isotropic $SL(2,\mathbb{R})$-invariant subspaces of Forni's space}


\author{Carlos Matheus\affil{1} and Gabriela  Weitze-Schmith\"usen\affil{2}}
\abbrevauthor{C. Matheus and G. Weitze-Schmith\"usen}
\headabbrevauthor{C. Matheus, and G. Weitze-Schmith\"usen}

\address{%
\affilnum{1}Universit\'e Paris 13, Sorbonne Cit\'e, LAGA, CNRS (UMR 7539), F- 93430, Villetaneuse, France
and
\affilnum{2}Institute for Algebra and Geometry , Karlsruhe Institute of Technology, 76128 Karlsruhe, Germany}

\correspdetails{weitze-schmithuesen@kit.edu}

\received{}
\revised{}
\accepted{}

\communicated{}

\begin{abstract}
We construct some examples of origamis (square-tiled surfaces) such that the Hodge bundles over the corresponding $SL(2,\mathbb{R})$-orbits on the moduli space admit non-trivial isotropic $SL(2,\mathbb{R})$-invariant subbundles. This answers a question posed to the authors by A. Eskin and G. Forni.
\end{abstract}

\maketitle

\section{Introduction}

The investigation of $SL(2,\mathbb{R})$-invariant probabilities on moduli spaces of Abelian (and/or quadratic) differentials on Riemann surfaces is a fascinating subjects whose applications nowadays include: the description of deviations of ergodic averages of interval exchange transformations, translation flows and billiards (see \cite{Zorich94} and \cite{Forni}), the confirmation of a conjecture of the physicists J. Hardy and J. Weber on the abnormal rate of diffusion of trajectories on typical realisations of Ehrenfest's wind-tree model of Lorenz gases (see \cite{DHL}), and the classification of the commensurability classes of all presently known non-arithmetic ball quotients (see \cite{KappesMoller}). 

After the recent breakthrough work of A. Eskin and M. Mirzakhani \cite{EM}, we have a better understanding of the geometry of $SL(2,\mathbb{R})$-invariant probabilities on moduli spaces of Abelian differentials. Indeed, A. Eskin and M. Mirzakhani showed the ``Ratner theory like statement'' that such measures are always supported on \emph{affine} suborbifolds of the moduli space. In their (long) way to prove this statement, they employed several different arguments inspired from several sources such as the \emph{low entropy method} of M. Einsiedler, A. Katok and E. Lindenstrauss \cite{EKL} and the \emph{exponential drift argument} of Y. Benoist and J. F. Quint \cite{BQ}. Moreover, as an important \emph{preparatory} step for the exponential drift argument, A. Eskin and M. Mirzakhani showed the \emph{semisimplicity} of the so-called \emph{Kontsevich-Zorich cocycle} on the Hodge bundle. 

The derivation in \cite{EM} of this semisimplicity property uses the work of G. Forni \cite{Forni} (see also \cite{FMZ2}) and the study of \emph{symplectic} and \emph{isotropic} $SL(2,\mathbb{R})$-invariant subbundles of the Hodge bundle. Interestingly enough, while \emph{symplectic} $SL(2,\mathbb{R})$-invariant subbundles occur in several known examples (see, e.g., \cite{Bouw:Moeller}, \cite{EKZ:cyclic}, \cite{FMZ2} and \cite{FMZ3}), it is not so easy to put the hands on concrete examples of \emph{isotropic} $SL(2,\mathbb{R})$-invariant subbundles. In fact, the only ``clue'' one had so far was that isotropic $SL(2,\mathbb{R})$-invariant subbundles are confined to the so-called \emph{Forni subbundle} of the Hodge bundle (see \cite{EM}). 

Partly motivated by the scenario of the previous paragraph, A. Eskin and G. Forni (independently) asked us about the actual \emph{existence} of concrete examples of $SL(2,\mathbb{R})$-invariant probabilities on some moduli spaces of Abelian differentials whose Hodge bundles admit \emph{non-trivial} $SL(2,\mathbb{R})$-invariant \emph{isotropic} subbundles.

In this note, we answer affirmatively the question of A. Eskin and G. Forni by exhibiting concrete \emph{square-tiled surfaces} (\emph{origamis}) such that the Hodge bundle over their $SL(2,\mathbb{R})$-orbits have non-trivial $SL(2,\mathbb{R})$-invariant isotropic subbundles. The idea of our construction is very simple:
\begin{itemize}
\item we start with the so-called \emph{Eierlegende Wollmilchsau} origami (see \cite{ForniSurvey} and \cite{HerrlichSchmithuesen}); as it was shown in the \cite{MY}, the Kontsevich-Zorich cocycle over the $SL(2,\mathbb{R})$-orbit $SL(2,\mathbb{R})\cdot (M_3,\omega_3)$ of the Eierlegende Wollmilchsau $(M_3,\omega_3)$ acts by a (very explicit) \emph{finite} group of symplectic matrices on a certain $SL(2,\mathbb{R})$-invariant subbundle of the Hodge bundle over $\mathcal{C}=SL(2,\mathbb{R})\cdot (M_3,\omega_3)$; 
\item in particular, by taking an adequate \emph{abstract} \emph{finite} cover $\widehat{\mathcal{C}}$ of $\mathcal{C}$, one eventually gets that the lift of the Kontsevich-Zorich cocycle to $\widehat{\mathcal{C}}$ acts \emph{trivially} on a certain $SL(2,\mathbb{R})$-invariant subbundle $H$ of the Hodge bundle over $\widehat{\mathcal{C}}$; it follows that \emph{any} $1$-dimensional (equivariant) subbundle of $H$ is an \emph{isotropic} $SL(2,\mathbb{R})$-invariant subbundle of the Hodge bundle over $\widehat{\mathcal{C}}$; however, there is no \emph{a priori} reason that $\widehat{\mathcal{C}}$ corresponds to the support of a $SL(2,\mathbb{R})$-invariant probability in some moduli space of Abelian differentials; 
\item to overcome the difficulty related to the fact that $\widehat{\mathcal{C}}$, we use the results of \cite{Schmithusen} (see also \cite{EllMcR}) to show that \emph{some} of the abstract covers $\widehat{\mathcal{C}}$ described in the previous item can be realised as $SL(2,\mathbb{R})$-orbits of certain square-tiled surfaces.
\end{itemize}

We organise this note as follows. In the next section, we will briefly recall the basic notions of translation surfaces, square-tiled surfaces, Kontsevich-Zorich cocycle and affine diffeomorphisms. Then, in the last section, we state and prove our main results, namely, Theorem \ref{t.F-iso} and \ref{thm-orni}, answering to the question of A. Eskin and G. Forni. Finally, we depict in Appendix \ref{a.A} our ``smallest'' origami satisfying the conclusions of Theorem \ref{t.F-iso}.


\section{Preliminaries}\label{s.preliminaries} 

The basic references for this entire section are the surveys of A. Zorich \cite{ZoSurvey}, and P. Hubert and T. Schmidt \cite{HubSchm}. Also, the reader may find useful to consult the introduction of the article \cite{MY} for further comments on the relationship between the Kontsevich-Zorich cocycle and the action on homology of affine diffeomorphisms of translation surfaces. 

A \emph{translation surface} is the data $(M,\omega)$ of a non-trivial Abelian differential $\omega$ on a Riemann surface $M$. This nomenclature comes from the fact that the local primitives of $\omega$ outside the set $\Sigma$ of its zeroes provides an atlas on $M-\Sigma$ whose changes of coordinates are all translations of the plane $\mathbb{R}^2$. In the literature, these charts are called \emph{translation charts} and an atlas formed by translation charts is called \emph{translation atlas} or \emph{translation (surface) structure}. For later use, we define the \emph{area} $a(M,\omega)$ of $(M,\omega)$ as $a(M,\omega):=(i/2)\int_M\omega\wedge\overline{\omega}$. 

The \emph{Teichm\"uller space} $\widehat{\mathcal{H}_g}$ \emph{of unit area Abelian differentials
 of genus} $g\geq 1$ is the set of unit area translation surfaces $(M,\omega)$ of genus $g\geq 1$ modulo the natural action of the group $\textrm{Diff}_0^+(M)$ of orientation-preserving homeomorphisms of $M$ isotopic to the identity, where $M$ is a fixed topological surface. The  \emph{moduli space} $\mathcal{H}_g$  \emph{of unit area Abelian differentials of genus} $g\geq 1$ is the set of unit area translation surfaces $(M,\omega)$ of genus $g\geq 1$ modulo the natural action of the group $\textrm{Diff}^+(M)$ of orientation-preserving homeomorphisms of $M$, where $M$ is a fixed topological surface. In particular, $\mathcal{H}_g=\widehat{\mathcal{H}_g}/\Gamma_g$ where $\Gamma_g:=\textrm{Diff}^+(M)/\textrm{Diff}_0^+(M)$ is the \emph{mapping class group} (of isotopy classes of orientation-preserving homeomorphisms of $M$).

The point of view of translation structures is useful because it makes clear that $SL(2,\mathbb{R})$ acts on the set of Abelian differentials $(M,\omega)$: indeed, given $h\in SL(2,\mathbb{R})$, we define $h\cdot(M,\omega)$ as the translation surface whose translation charts are given by post-composing the translation charts of $(M,\omega)$ with 
$h$. This action of $SL(2,\mathbb{R})$ descends to $\widehat{\mathcal{H}_g}$ and $\mathcal{H}_g$. The action of the diagonal subgroup $g_t:=\textrm{diag}(e^t, e^{-t})$ of $SL(2,\mathbb{R})$ is the so-called \emph{Teichm\"uller (geodesic) flow}. 

\begin{remark} By collecting together unit area Abelian differentials with orders of zeroes prescribed by a list $\kappa=(k_1,\dots,k_s)$ of positive integers with $\sum k_n=2g-2$, we obtain a subset $\mathcal{H}(\kappa)$ of $\mathcal{H}_g$ called \emph{stratum} in the literature. From the definition of the $SL(2,\mathbb{R})$-action on $\mathcal{H}_g$, it is not hard to check that the strata $\mathcal{H}(\kappa)$ are $SL(2,\mathbb{R})$-invariant.
\end{remark}

The \emph{Hodge bundle} $H_g^1$ over $\mathcal{H}_g$ is the quotient of the trivial bundle $\widehat{\mathcal{H}_g}\times H_1(M,\mathbb{R})$ by the natural action of the mapping-class group $\Gamma_g$ on \emph{both} factors. In this language, the \emph{Kontsevich-Zorich cocycle} $G_t^{KZ}$ is the quotient of the trivial cocycle $\widehat{G_t^{KZ}}:\widehat{\mathcal{H}_g}\times H_1(M,\mathbb{R})\to \widehat{\mathcal{H}_g}\times H_1(M,\mathbb{R})$
$$\widehat{G_t^{KZ}}(\omega,[c])=(g_t(\omega),[c])$$
by the mapping-class group $\Gamma_g$. In the sequel, we will call $G_t^{KZ}$ as KZ cocycle for short. 

For the sake of this note, let us restrict ourselves to the class of translation surfaces $(M,\omega)$ covering (with at most one ramification point) the square flat torus $\mathbb{T}^2=\mathbb{R}^2/\mathbb{Z}^2$  equipped with the Abelian differential induced by $dz$ on $\mathbb{C}=\mathbb{R}^2$. In the literature, these translation surfaces $(M,\omega)$ are called \emph{square-tiled surfaces} or \emph{origamis}. 

The stabiliser $SL(M,\omega)$ -- also known as \emph{Veech group} -- of a square-tiled surface $(M,\omega)\in\mathcal{H}_g$ with respect to the action of $SL(2,\mathbb{R})$ is commensurable to $SL(2,\mathbb{Z})$, and its $SL(2,\mathbb{R})$-orbit is a \emph{closed} subset of $\mathcal{H}_g$ isomorphic to the unit cotangent bundle $SL(2,\mathbb{R})/SL(M,\omega)$ of the hyperbolic surface $\mathbb{H}/SL(M,\omega)$. 

The Veech group $SL(M,\omega)$ consists of the ``derivatives'' (linear parts) of all \emph{affine diffeomorphisms} of $(M,\omega)$, that is, the orientation-preserving homeomorphisms of $M$ fixing the set $\Sigma$ of zeroes of $\omega$ whose local expressions in the translation charts of $(M,\omega)$ are affine maps of the plane. The group of affine diffeomorphisms of $(M,\omega)$ is denoted by $\textrm{Aff}(M,\omega)$ and it is possible to show that $\textrm{Aff}(M,\omega)$ is precisely the subgroup of elements of $\Gamma_g$ stabilising $SL(2,\mathbb{R})\cdot (M,\omega)$ in $\mathcal{H}_g$. The Veech group and the affine diffeomorphisms group are part of the following exact sequence
$$\{id\}\to \textrm{Aut}(M,\omega)\to \textrm{Aff}(M,\omega)\to SL(M,\omega)\to\{id\}$$
where, by definition, $\textrm{Aut}(M,\omega)$ is the subgroup of \emph{automorphisms} of $(M,\omega)$, i.e., the subgroup of elements of $\textrm{Aff}(M,\omega)$ whose linear part is trivial (i.e., identity). 

In this language, the KZ cocycle on the Hodge bundle over the $SL(2,\mathbb{R})$-orbit of $(M,\omega)$ is intimately related to the action on homology of $\textrm{Aff}(M,\omega)$. Indeed, since $\textrm{Aff}(M,\omega)\subset \Gamma_g$ is the stabiliser of $SL(2,\mathbb{R})\cdot (M,\omega)$ in $\mathcal{H}_g=\widehat{\mathcal{H}_g}/\Gamma_g$, we have that the KZ cocycle is the quotient of the trivial cocycle 
$$g_t\times id: \widehat{\mathcal{H}_g}\times H_1(M,\mathbb{R})\to \widehat{\mathcal{H}_g}\times H_1(M,\mathbb{R})$$
by $\textrm{Aff}(M,\omega)$.

For later use, we observe that, given a square-tiled surface $p:(M,\omega)\to(\mathbb{T}^2,dz)$ (where $p$ is a finite cover ramified precisely over $0\in\mathbb{T}^2$), the KZ cocycle, or equivalently $\textrm{Aff}(M,\omega)$, preserves the decomposition 
$$H_1(M,\mathbb{R})=H_1^{st}\oplus H_1^{(0)}(M,\mathbb{R}),$$
where  $H_1^{(0)}(M,\mathbb{R}):=\textrm{Ker}(p_*)$ and $H_1^{st}:=(p_*)^{-1}(H_1(\mathbb{T}^2,\mathbb{R}))$.
Here we denote by $(p_*)^{-1}(H_1(\mathbb{T}^2,\mathbb{R}))$ the isomorphic preimage of 
$H_1(\mathbb{T}^2,\mathbb{R})$ via $p_*$ which is the orthogonal to $H_1^{(0)}(M,\mathbb{R})$ with respect to the
intersection form.

Closing this preliminary section, we recall that, given a finite ramified covering $X_1\to X_2$ of Riemann surfaces, the \emph{ramification data} of a point $p\in X_2$ is the list of ramification indices of all pre-images of $p$ counted with multiplicities.

\section{Isotropic \boldmath{$SL(2,\mathbb{R})$}-invariant subbundles of Hodge bundle}\label{s.examples}

This section is divided into three parts. In Subsection \ref{ss.2.1}, we will show an ``abstract'' criterion (cf. Proposition \ref{p.2} below) leading to square-tiled surfaces such that the Hodge bundle over their $SL(2,\mathbb{R})$-orbits have $SL(2,\mathbb{R})$-invariant isotropic subbundles. Then, in  Subsection \ref{ss.2.2}, we will use the ``abstract'' criterion to exhibit (cf. Definition \ref{bigone}) a square-tiled surface with genus 15 and 512 squares covering the special (\emph{Eierlegende Wollmilchsau}) origami $(M_3,\omega_{(3)})$ such that the Hodge bundle over its $SL(2,\mathbb{R})$-orbit has $SL(2,\mathbb{R})$-invariant isotropic subbundles.  Finally, in Subsection \ref{ss.Orni} we present an infinite family of origamis with this
property. They come from coverings to an origami $(M_4,\omega_{(4)})$ also called {\em ornithorynque origami} having
similar properties as $(M_3, \omega_{(3)})$.

\subsection{``Abstract'' examples of  \boldmath{$SL(2,\mathbb{R})$}-invariant isotropic subbundles}\label{ss.2.1}
Let $(M_3, \omega_{(3)})$ be the \emph{Eierlegende Wollmilchsau} origami, that is, the translation surface associated to the Riemann surface $M_{3}$ defined by the algebraic equation
$$\{y^4=x^3-x\}$$
and the Abelian differential $\omega_{(3)}=c^{-1}dx/y^2$ where $c=\int_1^{\infty}\frac{dx}{\sqrt{x^3-x}}=\frac{(\Gamma(1/4))^2}{2\sqrt{2\pi}}$. Note that $M_{3}$ is a genus $3$ Riemann surface and $\omega_{(3)}$ has 4 simple zeroes at the points $x_1,\dots, x_4\in M_{3}$ over $0,1,-1,\infty$.

The square-tiled surface structure on $(M_{3},\omega_{(3)})$ becomes apparent from the natural covering $h(x,y)=(x,y^2)=(x,z)$ mapping $(M_{3},\omega_{(3)})$ into the torus (elliptic curve) 
$$T=\{z^2=x^3-x\}$$
equipped with the (unit area) Abelian differential $c^{-1}dx/z$. 

The KZ cocycle over the $SL(2,\mathbb{R})$-orbit of $(M_{3},\omega_{(3)})$, or more precisely, the homological action of the group $\textrm{Aff}(M_{3},\omega_{(3)})$ of affine diffeomorphisms of $(M_{3},\omega_{(3)})$ was analysed in details in \cite{MY}. Here, we will need the following fact proved there. Denote by $\textrm{Aff}_{(1)}(M_3,\omega_{(3)})$ the subgroup of affine elements fixing \emph{each} zero of $\omega_{(3)}$ and consider the canonical derivative morphism from $\textrm{Aff}(M_{3},\omega_{(3)})$ to the Veech group $SL(2,\mathbb{Z})$ of $(M_{3},\omega_{(3)})$. In this setting, let $\Gamma_{EW}(4)$ be the subgroup consisting of elements of  $\textrm{Aff}_{(1)}(M_{3},\omega_{(3)})$ whose image in $SL(2,\mathbb{Z})$ under the derivative morphism belong to the principal congruence subgroup $\Gamma(4)$ of level $4$ of $SL(2,\mathbb{Z})$. 

\begin{proposition}[cf. Proposition 7.1 of \cite{HooperWeiss}, or Lemma 2.8 of \cite{MY}]\label{p.1}
The elements of $\Gamma_{EW}(4)$ act by $\pm \textrm{id}$ on $H_1^{(0)}(M_{3},\mathbb{R})$.
\end{proposition}


Using this fact, one can build examples of $SL(2,\mathbb{R})$-invariant isotropic subbundles of the Hodge bundle 
by taking \emph{adequate} finite (ramified) coverings of the Eierlegende Wollmilchsau $(M_3,\omega_{(3)})$ 
(see Figure~\ref{ews}).

\begin{figure}[h]
\begin{center}
  \setlength{\unitlength}{1.3cm}
  \begin{picture}(6,2)
    \put(0,0){\framebox(1,1){1}}
    \put(1,0){\framebox(1,1){i}}
    \put(2,0){\framebox(1,1){-1}}
    \put(3,0){\framebox(1,1){-i}}
    \put(3,1){\framebox(1,1){-k}}
    \put(4,1){\framebox(1,1){-j}}
    \put(5,1){\framebox(1,1){k}}
    \put(6,1){\framebox(1,1){j}}
    \put(-.11,-.11){\framebox(.22,.22){}}
    \put(1.89,-.11){\framebox(.22,.22){}}
    \put(3.89,-.11){\framebox(.22,.22){}}
    \put(3.89,1.89){\framebox(.22,.22){}}
    \put(5.89,1.89){\framebox(.22,.22){}}
    \put(1,0){\circle{.2}}
    \put(3,0){\circle{.2}}
    \put(3,2){\circle{.2}}
    \put(5,2){\circle{.2}}
    \put(7,2){\circle{.2}}
    \put(0,1){\circle*{.2}}
    \put(2,1){\circle*{.2}}
    \put(4,1){\circle*{.2}}
    \put(6,1){\circle*{.2}}    
    \put(4.89,0.89){\rule{2.2mm}{2.2mm}}
    \put(0.89,0.89){\rule{2.2mm}{2.2mm}}
    \put(2.89,0.89){\rule{2.2mm}{2.2mm}}
    \put(6.89,0.89){\rule{2.2mm}{2.2mm}}
    \put(.45,.85){\line(1,2){.15}}
    \put(6.45,.85){\line(1,2){.15}}
    \put(1.44,.85){\line(1,2){.15}}
    \put(1.5,.85){\line(1,2){.15}}
    \put(5.44,.85){\line(1,2){.15}}
    \put(5.5,.85){\line(1,2){.15}}
    \put(2.4,.85){\line(1,2){.15}}
    \put(2.45,.85){\line(1,2){.15}}
    \put(2.5,.85){\line(1,2){.15}}
    \put(4.4,.85){\line(1,2){.15}}
    \put(4.45,.85){\line(1,2){.15}}
    \put(4.5,.85){\line(1,2){.15}}
    \put(3.55,1.85){\line(-1,2){.15}}
    \put(1.55,-.15){\line(-1,2){.15}}
    \put(4.5,1.85){\line(-1,2){.15}}
    \put(4.57,1.85){\line(-1,2){.15}}
    \put(0.5,-.15){\line(-1,2){.15}}
    \put(0.57,-.15){\line(-1,2){.15}}
    \put(5.5,1.85){\line(-1,2){.15}}
    \put(5.55,1.85){\line(-1,2){.15}}
    \put(5.6,1.85){\line(-1,2){.15}}
    \put(3.5,-.15){\line(-1,2){.15}}
    \put(3.55,-.15){\line(-1,2){.15}}
    \put(3.6,-.15){\line(-1,2){.15}}
    \put(6.47,1.85){\line(-1,2){.15}}
    \put(6.53,1.85){\line(-1,2){.15}}
    \put(6.58,1.85){\line(-1,2){.15}}
    \put(6.63,1.85){\line(-1,2){.15}}
    \put(2.48,-.15){\line(-1,2){.15}}
    \put(2.53,-.15){\line(-1,2){.15}}
    \put(2.58,-.15){\line(-1,2){.15}}
    \put(2.63,-.15){\line(-1,2){.15}}

  \end{picture}
\caption{The origami ($M_3$, $\omega_{(3)}$)}\label{ews}
\end{center}
\end{figure}
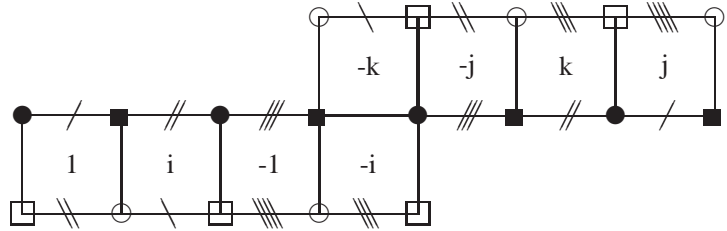

More precisely, the Eierlegende Wollmilchsau $(M_3,\omega_{(3)})$ can be thought as the collection of $8$ unit squares 
$sq(g)$ indexed by the elements $g\in Q=\{\pm1,\pm i,\pm j, \pm k\}$ of quaternion group glued together by the following rule: the right-hand side of $sq(g)$ is glued by translation to the left-hand side of $sq(gi)$ and the top side of $sq(g)$ is glued by translation to the bottom side of $sq(gj)$. In this language, we have a natural projection $\pi:(M_3,\omega_{(3)})\to\mathbb{R}^2/\mathbb{Z}^2$ of degree 8.
Consider now $p:(X,\omega)\to (M_3,\omega_{(3)})$ a finite ramified covering and let's denote by $q:(X,\omega)\to 
E = \mathbb{R}^2/\mathbb{Z}^2$ the natural covering $q = 
\pi\circ p$. 

The next proposition puts constraints on the group $\textrm{Aff}(X,\omega)$ of affine diffeomorphisms of $(X,\omega)$ depending on the ramification data at certain points of $(M_3,\omega_{(3)})$.

\begin{proposition}\label{p.2} Consider the following conditions on $(X,\omega)$:
\begin{itemize}
\item[A)] 
  The covering $q:(X,\omega)\to\mathbb{R}^2/\mathbb{Z}^2 = E$ is ramified at most 
  over 4-division points.
\item[B)] 
  The covering $q$  is ramified above  
  $(0,0)$, $(3/4,0)$ and $(0,1/4)$. The ramification data 
  at $(0,0)$, $(3/4,0)$ and $(0,1/4)$ are distinct and different from
  all other ramification data.
\item[C)] 
  The ramification data of $p:(X,\omega)\to (M_3,\omega_{(3)})$ at the four zeroes of $\omega_{(3)}$ are distinct.
\end{itemize}
It holds:
\begin{itemize}
\item[i)]
  If $(X,\omega)$ satisfies A) and B), then any  affine diffeomorphism $f\in\textrm{Aff}(X,\omega)$
  descends via $p$ to an affine diffeomorphism $g\in\textrm{Aff}(M_3,\omega_{(3)})$ (i.e., $g\circ p = p\circ f$).
  Furthermore $D(f) = D(g)$ lies in $\Gamma(4)$.
\item[ii)]
  If $(X,\omega)$ in addition satisfies C), 
  then $g$ fixes the four zeroes 
  of $\omega_{(3)}$, that is, $g\in\textrm{Aff}_{(1)}(M_3,\omega_{(3)})$. We have in particular
  that $g$ lies in $\Gamma_{EW}(4)$.
\end{itemize}
\end{proposition}

The following basic fact about ramified coverings will be crucial 
in the proof of Proposition~\ref{p.2}. Let $p:X \to Y$
be a ramified covering of topological surfaces. If $f:X \to X$ 
is a homeomorphism of $X$ which descends to $g: Y \to Y$ via
$p$, i.e. 
$g \circ p = p \circ f$, then $f$ preserves ramification indices
and $g$ preserves ramification data.

\begin{proof}[Proof of Proposition \ref{p.2}]
{\bf i)} 
Suppose that A) and B) hold. Any affine diffeomorphism $f$
of $X$ descends via $q$
to some affine diffeomorphism $h:E \to E$ (see e.g. \cite[Prop. 2.6]{Schmithusen} 
for a detailed proof). It follows from A) and B) that $h$ fixes the point $\infty = (0,0)$, 
$(3/4,0)$ and $(0,1/4)$ pointwise, thus $D(f) = D(h)$
is in $\Gamma(4)$.\\
Since $h$ fixes $\infty$, it is also a homeomorphism of the 
once-punctured torus $E^* = E\backslash\{\infty\}$.
We have from \cite[Prop. 1.2]{HerrlichSchmithuesen}
that any affine diffeomorphism of $E$, which restricts to one of $E^*$,
lifts to $M_3$. More precisely,
since the degree 8 covering $M_3 \to E$ is normal, there are 8 lifts.
Again since this covering is normal, one of the lifts is a descend of $f$,
see e.g. \cite[Lemma 3.13]{MatheusSchmithuesen}.
Thus $f$ descends via $p$ to some affine homeomorphism $g$ of $(M_3, \omega_{(3)})$, 
and $D(g) = D(f) = D(h)$ is in $\Gamma(4)$.\\[2mm]
{\bf ii)}
Suppose that A), B) and C) hold. The descend $g$ found in i) 
is an affine homeomorphism and therefore permutes the zeroes of $\omega_{(3)}$.
It directly follows from C) that it even fixes them pointwise
and thus is in $\Gamma_{EW}(4)$.
\end{proof}


By putting together Propositions \ref{p.1} and \ref{p.2}, it is not difficult to get examples of non-trivial $SL(2,\mathbb{R})$-invariant isotropic subbundles of the Hodge bundle over $SL(2,\mathbb{R})$-orbits of origamis: 

\begin{theorem}\label{t.F-iso} Let $p:(X,\omega)\to(M_3,\omega_{(3)})$ be a finite ramified covering of the Eierlegende Wollmilchsau $(M_3,\omega_{(3)})$ satisfying the conditions A), B) and C) stated in Proposition \ref{p.2}. Then, the Hodge bundle over the $SL(2,\mathbb{R})$-orbit of $(X,\omega)$ contains non-trivial $SL(2,\mathbb{R})$-invariant isotropic subbundles. 
\end{theorem} 

\begin{proof} Given a finite ramified covering $p:(X,\omega)\to(M_3,\omega_{(3)})$, denote by $H:=p^{-1}(H_1^{(0)}(M_3,\mathbb{R}))$ the $4$-dimensional subbundle of the Hodge bundle over $SL(2,\mathbb{R})\cdot (X,\omega)$ obtained by lifting the $4$-dimensional subbundle $H_1^{(0)}(M_3,\mathbb{R})$ of the Hodge bundle over $SL(2,\mathbb{R})\cdot(M_3,\omega_{(3)})$.
Here similarly as before $p^{-1}(H_1^{(0)}(M_3,\mathbb{R}))$ denotes the isomorphic
preimage of $H_1^{(0)}(M_3,\mathbb{R})$ which is the intersection of the kernel of $q_*$ with the (symplectic) orthogonal to the kernel of $p_*$.

By definition, $SL(2,\mathbb{R})$ acts by post-composition with translation charts. Thus, we have that $H$ is a $SL(2,\mathbb{R})$-invariant subbundle because of the $SL(2,\mathbb{R})$-invariance of $H_1^{(0)}(M_3,\mathbb{R})$.

Now, suppose that $p:(X,\omega)\to(M_3,\omega_{(3)})$ verifies the hypothesis of the theorem. By Proposition \ref{p.2}, we know that any affine diffeomorphism $f\in\textrm{Aff}(X,\omega)$ descends to (a unique) $g\in\Gamma_{EW}(4)$. Therefore, by Proposition \ref{p.1}, we have that any $f\in\textrm{Aff}(X,\omega)$ acts \emph{trivially} (i.e., by 
$\pm \textrm{id}$) on the $SL(2,\mathbb{R})$-invariant subbundle $H$. 

In particular, \emph{every} equivariant subbundle $E\subset H$ over $SL(2,\mathbb{R})\cdot(X,\omega)$ is $SL(2,\mathbb{R})$-invariant. Thus, the Hodge bundle over $SL(2,\mathbb{R})\cdot(X,\omega)$ contains \emph{plenty} $SL(2,\mathbb{R})$-invariant \emph{isotropic} subbundles: for instance, any $1$-dimensional equivariant subbundle $E\subset H$ has this property. 
\end{proof}

\begin{remark} More examples of $SL(2,\mathbb{R})$-invariant isotropic subbundles of the Hodge bundle can be produced by taking adequate finite covers $p:(X,\omega)\to (M_4,\omega_{(4)})$ of an origami 
$(M_4,\omega_{(4)})$ of genus $4$ introduced in \cite{FM}, \cite{FMZ} (and sometimes called \emph{Ornithorynque} in the literature). Indeed, the homological action of $\textrm{Aff}(M_4,\omega_{(4)})$ was also studied in \cite{MY} where it is shown that any $f\in\textrm{Aff}(M_4,\omega_{(4)})$ with derivative in the principal congruence subgroup $\Gamma(3)$ of level $3$ of $SL(2,\mathbb{Z})$ acts trivially on $H_1^{(0)}(M_4,\mathbb{R})$ (see Subsection 3.5 and Lemma 3.3 of Subsection 3.7 of \cite{MY}). In particular, one setup conditions (\emph{similar} to A), B) and C) above) on the ramification data of $p:(X,\omega)\to (M_4,\omega_{(4)})$ over $3$-torsion points of $(M_4,\omega_{(4)})\to\mathbb{R}^2/\mathbb{Z}^2$ so that any $1$-dimensional equivariant subbundle $E$ of the $6$-dimensional subbundle $H:=p^{-1}(H_1^{(0)}(M_4,\mathbb{R}))$ is $SL(2,\mathbb{R})$-invariant and isotropic. This will be briefly described in Subsection \ref{ss.Orni} below. 
\end{remark}

In the next subsection, we will use Theorem \ref{t.F-iso} to produce an \emph{explicit} origami (square-tiled surface) $(X,\omega)$ with non-trivial $SL(2,\mathbb{R})$-invariant isotropic subbundle in the Hodge bundle over its $SL(2,\mathbb{R})$-orbit.

\subsection{A ``concrete'' example of \boldmath{$SL(2,\mathbb{R})$}-invariant isotropic subbundles}\label{ss.2.2}

We construct in the following an example $X$ of an origami which 
satisfies the conditions A), B) and C) of Proposition~\ref{p.2}. 
We start from the origami $(M_3,\omega_{(3)})$ shown in Figure~\ref{ews}.
Observe that this is the same origami as shown on the 
left hand side of Figure~\ref{denotations}.
Its horizontal and vertical
gluings are given by the two permutations
\begin{equation}\label{wms-permutations}
  \wmsa = (1,6,3,8)(2,5,4,7) \mbox{ and } \wmsb = (1,2,3,4)(5,6,7,8).
\end{equation}
It has the four zeroes $\circ$, $\bullet$, $\blacksquare$ and $\square$.
We subdivide each square of 
$(M_3,\omega_{(3)})$ into $16$ subsquares. 
This gives an origami with $128$ squares (see Figure~\ref{bauplan}). \\

The origami $X$ will be a degree 4 cover of it ramified over the four points 
$Q_1$, $Q_2$, $Q_3$ and $Q_4$  and over the four zeroes
of $\omega_{(3)}$. Here $Q_i$ ($i \in \{1, \ldots, 4\}$) 
is  the end point of the segment $e_i$ shown in Figure~\ref{bauplan}  which is not a zero. 
The ramification over $Q_i$
will be given by the permutation $\pi_i$ with
\[\pi_1 = (1,3)(2,4), \pi_2 = (1,2), \pi_3 = (1,3)(2,4) \mbox{ and } \pi_4 = (1,3,2)\] 
The origami $X$ thus will consist of 
$512 = 4\cdot 8 \cdot 16$ squares.\\  
More precisely 
we construct $X$ as follows (see Figure~\ref{all}). We take four copies of the origami shown in Figure~\ref{bauplan}.
We glue each edge to the corresponding edge in the same copy.
Only the four edges $e_1$, $e_2$, $e_3$ and $e_4$
are glued to the corresponding edges $\overline{e_1}$, $\overline{e_2}$,
$\overline{e_3}$ and $\overline{e_4}$, respectively, of a possibly different copy. 
The gluings for those are given by
the permutations $\pi_1$, $\pi_2$, $\pi_3$, $\pi_4$, respectively. 
This means
e.g. that the edge $e_1$ in Copy $1$ is glued to $\overline{e_1}$
in Copy $3 = \pi_1(1)$, $e_1$ in Copy $2$ is glued to $\overline{e_1}$
in Copy $4 = \pi_1(2)$, $e_1$ in Copy $3$ is glued to $\overline{e_1}$
in Copy $1 = \pi_1(3)$ and $e_1$ in Copy $4$ is glued to $\overline{e_1}$
in Copy $2 =  \pi_1(4)$.\\

\begin{figure}[htb!]
    \begin{center}
      \setlength{\unitlength}{2.4cm}
      \begin{picture}(2,6.8)
        \put(1,0){\framebox(1,1){}}
        \put(1,1){\framebox(1,1){}}
        \put(1,2){\framebox(1,1){}}
        \put(1,3){\framebox(1,1){}}
        \put(0,3){\framebox(1,1){}}
        \put(0,4){\framebox(1,1){}}
        \put(0,5){\framebox(1,1){}}
        \put(0,6){\framebox(1,1){}}
        \put(1.94,-.06){\framebox(.12,.12){}}
        \put(1.94,1.94){\framebox(.12,.12){}}
        \put(1.94,3.94){\framebox(.12,.12){}}
        \put(-.06,3.94){\framebox(.12,.12){}}
        \put(-.06,5.94){\framebox(.12,.12){}}
        \put(2,1){\circle{.2}}
        \put(0,3){\circle{.2}}
        \put(0,5){\circle{.2}}
        \put(0,7){\circle{.2}}
        \put(2,3){\circle{.2}}
        \put(1,0){\circle*{.2}}
        \put(1,2){\circle*{.2}}
        \put(1,4){\circle*{.2}}
        \put(1,6){\circle*{.2}}    
        \put(0.93,0.93){\rule{3mm}{3mm}}
        \put(0.93,2.93){\rule{3mm}{3mm}}
        \put(0.93,4.93){\rule{3mm}{3mm}}
        \put(0.93,6.93){\rule{3mm}{3mm}}
        \put(0.25,3){\line(0,1){4} }
        \put(0.5,3){\line(0,1){4} }
        \put(0.75,3){\line(0,1){4} }
        \put(1.25,0){\line(0,1){4} }
        \put(1.5,0){\line(0,1){4} }
        \put(1.75,0){\line(0,1){4} }
        \put(0,6.25){\line(1,0){1} }
        \put(0,6.5){\line(1,0){1} }
        \put(0,6.75){\line(1,0){1} }
        \put(0,5.25){\line(1,0){1} }
        \put(0,5.5){\line(1,0){1} }
        \put(0,5.75){\line(1,0){1} }
        \put(0,4.25){\line(1,0){1} }
        \put(0,4.5){\line(1,0){1} }
        \put(0,4.75){\line(1,0){1} }
        \put(0,3.25){\line(1,0){2} }
        \put(0,3.5){\line(1,0){2} }    
        \put(0,3.75){\line(1,0){2} }
        \put(1,.25){\line(1,0){1} }
        \put(1,.5){\line(1,0){1} }
        \put(1,.75){\line(1,0){1} }
        \put(1,1.25){\line(1,0){1} }
        \put(1,1.5){\line(1,0){1} }
        \put(1,1.75){\line(1,0){1} }
        \put(1,2.25){\line(1,0){1} }
        \put(1,2.5){\line(1,0){1} }
        \put(1,2.75){\line(1,0){1} }
        \put(2.1,0.5){$a$}
        \put(-.2,4.5){$a$}
        \put(2.1,1.5){$b$}
        \put(-.2,3.5){$b$}
        \put(2.1,2.5){$c$}
        \put(-.2,6.5){$c$}
        \put(2.1,3.5){$d$}
        \put(-.2,5.5){$d$}
        \put(1.1,4.5){$e$}
        \put(.85,2.5){$e$}
        \put(1.1,5.5){$f$}
        \put(.85,1.5){$f$}
        \put(1.1,6.5){$g$}
        \put(.85,0.5){$g$}
        \linethickness{2mm}
        \put(1.99,0.05){\line(0,1){.2}}
        \put(2.1,.1){\large $e_1$}        
        \put(2,.25){\large \circle*{.08}}      
        \put(1.83,.31){$Q_1$}
        \put(-.3,4.1){\large $\overline{e_1}$}
        \put(0,4.25){\large \circle*{.08}}      
        \put(0.02,4.31){$Q_1$}
        \put(-0.01,4.05){\line(0,1){.2}}
        \put(1.99,1.05){\line(0,1){.2}}
        \put(2.1,1.1){\large $e_2$}
        \put(2,1.25){\large \circle*{.08}}      
        \put(1.83,1.31){$Q_2$}
        \put(-.3,3.1){\large $\overline{e_2}$}
        \put(0,3.25){\large \circle*{.08}}      
        \put(0.02,3.31){$Q_2$}
        \put(-0.01,3.05){\line(0,1){.2}}
        \put(.99,4.05){\line(0,1){.2}}
        \put(.7,2.1){\large $\overline{e_3}$}

        \put(1,2.25){\large \circle*{.08}}      
        \put(1,2.3){$Q_3$}

        \put(1.1,4.1){\large $e_3$}
        \put(1,4.25){\large \circle*{.08}}      
        \put(.81,4.31){$Q_3$}
        \put(.99,2.05){\line(0,1){.2}}
        \put(1.75,-.01){\line(1,0){.2}}
        \put(1.75,4){\circle*{.08}}
        \put(1.58,3.88){$Q_4$}
        \put(1.75,4.1){\large $e_4$}
        \put(1.75,-.25){\large $\overline{e_4}$}
        \put(1.75,0){\circle*{.08}}
        \put(1.58,.05){$Q_4$}
        \put(1.75,3.95){\line(1,0){.2}}
        \linethickness{0.2mm}
      \end{picture}\\[5mm]
      \caption{Building plan for the origami $X$ from Definition~\ref{bigone}; 
        the point $Q_i$ respectively is the end point of $e_i$ which is not a singularity}\label{bauplan}
    \end{center}
  \end{figure}
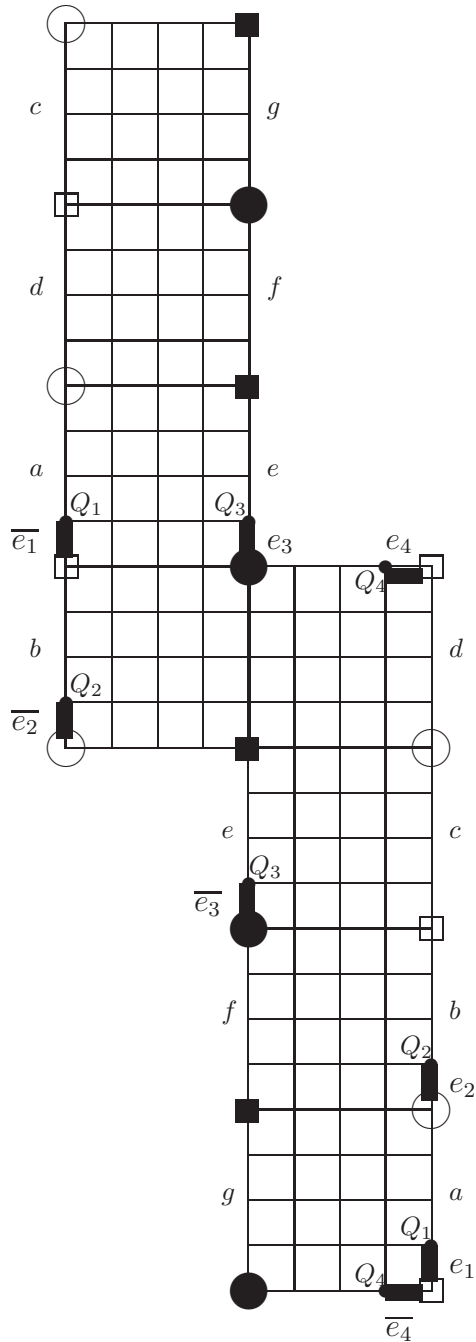


We will label the $512$
squares of $X$ by tuples 
$(a,i,j,k) \in \{1, \ldots, 8\} \times  \Z/4\Z \times  \Z/4\Z \times  \Z/4\Z$, 
where $a$ denotes the number
of the square in  $M_3$, $(i,j)$ denotes which of the $16$ subsquares it is,
and $k$ is the leaf of the cover (see Figure~\ref{denotations}).\\

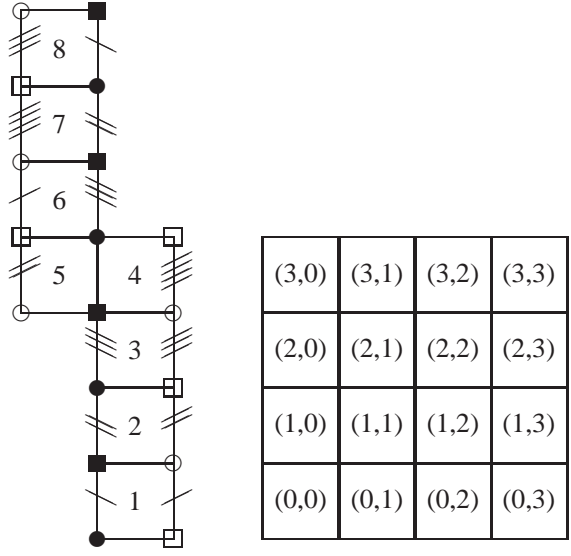
\begin{figure}[h]
  \begin{center}
    \setlength{\unitlength}{1cm}
    \begin{picture}(2,7)
      \put(1,0){\framebox(1,1){1}}
      \put(1,1){\framebox(1,1){2}}
      \put(1,2){\framebox(1,1){3}}
      \put(1,3){\framebox(1,1){4}}
      \put(0,3){\framebox(1,1){5}}
      \put(0,4){\framebox(1,1){6}}
      \put(0,5){\framebox(1,1){7}}
      \put(0,6){\framebox(1,1){8}}
    \put(1.85,.45){\line(2,1){.4}}
    \put(-.15,4.45){\line(2,1){.4}}
    \put(1.85,1.45){\line(2,1){.4}}
    \put(1.85,1.55){\line(2,1){.4}}
    \put(-.15,3.45){\line(2,1){.4}}
    \put(-.15,3.55){\line(2,1){.4}}
    \put(1.85,2.4){\line(2,1){.4}}
    \put(1.85,2.5){\line(2,1){.4}}
     \put(1.85,2.6){\line(2,1){.4}}  
    \put(-.15,6.4){\line(2,1){.4}}
    \put(-.15,6.5){\line(2,1){.4}}
    \put(-.15,6.6){\line(2,1){.4}}
    \put(1.85,3.3){\line(2,1){.4}}
    \put(1.85,3.4){\line(2,1){.4}}
    \put(1.85,3.5){\line(2,1){.4}}  
    \put(1.85,3.6){\line(2,1){.4}}
    \put(-.15,5.3){\line(2,1){.4}}  
    \put(-.15,5.4){\line(2,1){.4}}
    \put(-.15,5.5){\line(2,1){.4}}
    \put(-.15,5.6){\line(2,1){.4}}
    \put(.85,.65){\line(2,-1){.4}}
    \put(.85,6.65){\line(2,-1){.4}}
    \put(.85,1.65){\line(2,-1){.4}}
    \put(.85,1.55){\line(2,-1){.4}}
    \put(.85,5.65){\line(2,-1){.4}}
    \put(.85,5.55){\line(2,-1){.4}}
    \put(.85,2.6){\line(2,-1){.4}}
    \put(.85,2.7){\line(2,-1){.4}}
    \put(.85,2.8){\line(2,-1){.4}}
    \put(.85,4.6){\line(2,-1){.4}}
    \put(.85,4.7){\line(2,-1){.4}}
    \put(.85,4.8){\line(2,-1){.4}} 
    \put(2.4,2.85){\line(1,2){.15}}
    \put(2.45,2.85){\line(1,2){.15}}
    \put(2.5,2.85){\line(1,2){.15}}
    \put(.4,5.85){\line(1,2){.15}}
    \put(.45,5.85){\line(1,2){.15}}
    \put(.5,5.85){\line(1,2){.15}}
    \put(-.11,3.89){\framebox(.22,.22){}}
    \put(-.11,5.89){\framebox(.22,.22){}}
    \put(1.89,-.11){\framebox(.22,.22){}}
    \put(1.89,1.89){\framebox(.22,.22){}}
    \put(1.89,3.89){\framebox(.22,.22){}}
    \put(0,3){\circle{.2}}
    \put(0,5){\circle{.2}}
    \put(0,7){\circle{.2}}
    \put(2,1){\circle{.2}}
    \put(2,3){\circle{.2}}
    \put(1,0){\circle*{.2}}
    \put(1,2){\circle*{.2}}
    \put(1,4){\circle*{.2}}
    \put(1,6){\circle*{.2}}    
    \put(.89,0.89){\rule{2.2mm}{2.2mm}}
    \put(.89,2.89){\rule{2.2mm}{2.2mm}}
    \put(.89,4.89){\rule{2.2mm}{2.2mm}}
    \put(.89,6.89){\rule{2.2mm}{2.2mm}}
    \end{picture}
    \hspace*{10mm}
    \begin{picture}(4,4)
      \put(0,0){\framebox(1,1){(0,0)}}
      \put(1,0){\framebox(1,1){(0,1)}}
      \put(2,0){\framebox(1,1){(0,2)}}
      \put(3,0){\framebox(1,1){(0,3)}}
      \put(0,1){\framebox(1,1){(1,0)}}
      \put(1,1){\framebox(1,1){(1,1)}}
      \put(2,1){\framebox(1,1){(1,2)}}
      \put(3,1){\framebox(1,1){(1,3)}}
      \put(0,2){\framebox(1,1){(2,0)}}
      \put(1,2){\framebox(1,1){(2,1)}}
      \put(2,2){\framebox(1,1){(2,2)}}
      \put(3,2){\framebox(1,1){(2,3)}}
      \put(0,3){\framebox(1,1){(3,0)}}
      \put(1,3){\framebox(1,1){(3,1)}}
      \put(2,3){\framebox(1,1){(3,2)}}
      \put(3,3){\framebox(1,1){(3,3)}}
    \end{picture}
    \caption{
      Left side: Labels $a$ denoting the square in $M_3$.
      Right side: labels $(i,j)$ denoting the position within the ($4\times 4$)-squares.}
    \label{denotations}
\end{center}
\end{figure}

\begin{definition}\label{bigone}
In the following we use the two permutations $\wmsa $ and $\wmsb$ in $S_8$
defined in (\ref{wms-permutations}).
Let  $X$ be the origami defined by the following two permutations 
(compare Figure~\ref{bauplan} and Figure~\ref{denotations}):  
\[
  \begin{array}{ll}
   \sigma_a:& (a,i,j,k) \mapsto 
   \begin{cases}
     (a,i,j+1,k) \mbox{ if } j \in \{0,1,2\},\\
     (\wmsa(a),i,j+1,k) 
     \begin{array}[t]{l}
       \mbox{ if } j = 3\\  
       \hspace*{-2,2cm} \mbox{ and } 
       (a,i,j) \not\in \{(1,0,3),(2,0,3),(6,0,3)\},
     \end{array}\\
     (\wmsa(a),i,j+1,\pi_1(k)) \mbox{ if }  (a,i,j) = (1,0,3)\\
     (\wmsa(a),i,j+1,\pi_2(k)) \mbox{ if }  (a,i,j) = (2,0,3)\\ 
      (\wmsa(a),i,j+1,\pi_3(k)) \mbox{ if }  (a,i,j) = (6,0,3)\\ 
   \end{cases}\\[2cm]
   \sigma_b:& (a,i,j,k) \mapsto 
   \begin{cases}
     (a,i+1,j,k) \mbox{ if } i \in \{0,1,2\}\\
     (\wmsb(a),i+1,j,k) 
     \begin{array}[t]{l}
       \mbox{ if } i = 3 \\ 
       \hspace*{-2,2cm} \mbox{ and } (a,i,j) \not\in \{(4,3,3)\}
       \end{array}\\ 
      (\wmsb(a),i+1,j,\pi_4(k)) \mbox{ if }  (a,i,j) = (4,3,3)
   \end{cases}\\

  \end{array}
\]

\end{definition}

We now list some properties of the origami $X$.

\begin{lemma}
Let $X$ be the origami defined in Definition~\ref{bigone}. Recall
that $Q_i$ ($i \in \{1,2,3,4\}$) is the end point of $e_i$ which is not a zero (see Figure~\ref{bauplan}).
Then we have:
\begin{enumerate}
\item[i)] 
  $X$ has genus 15.  It has 17 zeroes and lies in 
  $\cH(5,3,3,3,2,\underbrace{1,\ldots, 1}_{12})$.
\item[ii)]
  The ramification data of the degree $4$ cover $p:X \to M_3$ are:
  \begin{itemize}
    \item[]  $(2,2)$ at the zero $\bullet$,\;  $(2,1,1)$ at the zero $\circ$,
    \item[]  $(1,1,1,1)$ at the zero $\blacksquare$, \; $(3,1)$ at the zero $\square$
    \item[]  $(2,2)$ at $Q_1$,\; $(2,1,1)$ at $Q_2$, 
    \item[]  $(2,2)$ at $Q_3$\, and  $(3,1)$ at $Q_4$.
  \end{itemize}
\item[iii)]
  The ramification data of the degree $32$ cover $q:X \to E$ are:
  \begin{itemize}
  \item[]
    $(6,4,4,4,2,2,2,2,2,2,2)$ \, at the  point $\infty$,
  \item[]
    $(2,2,2,2,2,\underbrace{1,\ldots, 1}_{22})$ \, at  the point $(0,\frac{1}{4})$ and
  \item[]
    $(3, \underbrace{1, \ldots, 1}_{29})$ \, at the point $(\frac{3}{4},0)$
  \end{itemize}
\end{enumerate}
In particular $X$ satisfies the conditions A), B) and C)\, from Proposition~\ref{p.2}.
\end{lemma}

\begin{proof}
{\bf ii)}
We directly obtain from the construction of $X$
that the monodromy of a small positively oriented loop around $Q_i$ on $M_3$ is $\pi_i$.
Thus the ramification data of $p$ at $Q_i$ is the list
of the lengths of cycles of $\pi_i$. A small loop around the zero $\bullet$
on $M_3$ crosses the slit $\overline{e_3}$, thus the monodromy
is $\pi_3^{-1} = (1,3)(2,4)$. Similarly one obtains for
the loops around  $\circ$, $\blacksquare$ and $\square$
the monodromies $\pi_2^{-1} = (1,2)$, the identity $id$ and ${\pi_1}^{-1}\pi_4^{-1} = (1,4,2)$,
respectively.
This gives the ramification data of $p$ at the zeroes. \\
{\bf iii)}
Recall that for the composition $q = (p :X \stackrel{4}{\to} M_3)\circ (\pi: M_3 \stackrel{8}{\to} E)$ we have:
for each  $x$ in $X$ the ramification index of $q$ at $x$ is equal to the ramification
index of $p$ at $x$ multiplied by the ramification index of $\pi$ at $p(x)$. 
Recall that $\infty$ has the four zeroes $\bullet$, $\circ$, $\blacksquare$ and $\square$
as preimages on $M_3$ and that their ramification index with respect to $\pi$
is $2$. The point
$(0,1/4)$ has the three preimages $Q_1$, $Q_2$ and $Q_3$ and further five
preimages over which $p$ is unramified (see Figure~\ref{bauplan}).
Finally, $(3/4,0)$ has the preimage $Q_4$ and seven
preimages which are not ramification points of $p$. 
Now we obtain the data in iii) from the data in ii).\\
{\bf i)} follows from iii).
\end{proof}

\subsection{Another ``concrete'' example of \boldmath {$SL(2,\mathbb{R})$}-invariant isotropic subbundles}\label{ss.Orni}%

In this section we describe how one can similarly obtain non-trivial isotropic $SL(2,\mathbb{R})$-invariant subbundles
using coverings of the Ornithorynque origami. Recall that the Ornithorynque origami $\Orni$
is the origami shown in Figure~\ref{figure-orni}. It carries a natural covering to the torus
$E$ of degree 3 indicated in the figure (where we take $E$ as the torus obtained from a
$2 \times 2$ square) 
which we denote by 
\[\pi_2: \Orni \to E.\]
$\Orni$ has the three singularities
$\textbf{X}_1 = \circledcirc$,  $\textbf{Y}_1 = \bullet$  and $\textbf{Z}_1 = \circ$, see Figure~\ref{figure-orni},
and is of genus 4.
Its Veech group again is the full group $\SLZ$, see \cite[p. 473]{MY}.  
Similarly as before we use as a main tool that we explicitly know from \cite[Section 3]{MY}  
the elements of the affine group of $\Orni$ that act trivially on the subbundle $H:=H_1^{(0)}(M_4,\mathbb{R})$. 
More precisely, we have that
$f$ acts trivially on $H$ if and only if   
its derivative $D(f)$ lies in the principal congruence group $\Gamma(3)$ 
and $f$ fixes all preimages of the 2-division 
points on $E$ under $\pi_2$.
Here we choose the point $A$ on the torus $E$ shown in Figure~\ref{rampoints} as zero point
and call the other 2-division points $\textbf{X}$, $\textbf{Y}$ and $\textbf{Z}$, see again Figure~\ref{rampoints}. 
In particular, we have $\pi_2(\textbf{X}_1) = \textbf{X}$, $\pi_2(\textbf{Y}_1) = \textbf{Y}$ and $\pi_2(\textbf{Z}_1) = \textbf{Z}$.
Thus \cite{MY}  gives us:
\begin{equation}\label{orni-trivial-homology}
\begin{array}{l}
f \mbox{ acts trivially on } H:=H_1^{(0)}(M_4,\mathbb{R}) \Leftrightarrow\\ 
\hspace*{3mm} D(f) \in \Gamma(3) \mbox{ and } f(P) = P  \mbox{ for all } P \in  
\{A_1, A_2, A_3, \textbf{X}_1, \textbf{Y}_1, \textbf{Z}_1\}
\end{array}
\end{equation}
Here $A_1$, $A_2$ and $A_3$ are the three preimages of $A$ (see Figure~\ref{figure-orni}).
The fact that the three unramified points $A_1$, $A_2$ and $A_3$
have to be fixed pointwise will give us some extra trouble in our construction which enforces
an additional step: we will use a covering $h$
of $\Orni$ which is ramified over them and takes care of this problem.

\begin{figure}[htb!]
  \begin{center}
    \setlength{\unitlength}{1.3cm}
    \begin{picture}(8,2.5)
      \put(0,0){\framebox(1,1){8}}
      \put(1,0){\framebox(1,1){9}}
      \put(3,0){\framebox(1,1){10}}
      \put(4,0){\framebox(1,1){11}}
      \put(6,0){\framebox(1,1){12}}
      \put(7,0){\framebox(1,1){7}}
      \put(0,1){\framebox(1,1){2}}
      \put(1,1){\framebox(1,1){1}}
      \put(3,1){\framebox(1,1){6}}
      \put(4,1){\framebox(1,1){5}}
      \put(6,1){\framebox(1,1){4}}
      \put(7,1){\framebox(1,1){3}}
      \put(0.4,-.3){\scs{6}}
      \put(1.4,-.3){\scs{3}}
      \put(3.4,-.3){\scs{4}}
      \put(4.4,-.3){\scs{1}}
      \put(6.4,-.3){\scs{2}}
      \put(7.4,-.3){\scs{5}}
      \put(0.35,2.05){\scs{12}}
      \put(1.35,2.05){\scs{11}}
      \put(3.4,2.05){\scs{8}}
      \put(4.4,2.05){\scs{7}}
      \put(6.35,2.05){\scs{10}}
      \put(7.4,2.05){\scs{9}}
      \put(-.2,.4){\scs{7}}
      \put(2.8,.4){\scs{9}}
      \put(5.7,.4){\scs{11}}
      \put(2.05,.4){\scs{10}}
      \put(5.05,.4){\scs{12}}
      \put(8.05,.4){\scs{8}}
      \put(-.2,1.4){\scs{1}}
      \put(2.8,1.4){\scs{5}}
      \put(5.8,1.4){\scs{3}}
      \put(2.05,1.4){\scs{2}}
      \put(5.05,1.4){\scs{6}}
      \put(8.05,1.4){\scs{4}}
      \put(1.2,1.7){\vector(-1,0){.5}}
      \put(1.2,1.7){\scs{+1}}
      \put(4.2,1.7){\vector(-1,0){.5}}
      \put(4.2,1.7){\scs{+1}}
      \put(7.2,1.7){\vector(-1,0){.5}}
      \put(7.2,1.7){\scs{+1}}
      \linethickness{1mm}
      \put(1,1){\line(0,1){1}}
      \put(4,1){\line(0,1){1}}
      \put(7,1){\line(0,1){1}}
      \put(1,1){\circle*{.2}}
      \put(4,1){\circle*{.2}}
      \put(7,1){\circle*{.2}}
      \put(0,1){\circle{.2}}
      \put(3,1){\circle{.2}}
      \put(6,1){\circle{.2}}
      \put(2,1){\circle{.2}}
      \put(5,1){\circle{.2}}
      \put(8,1){\circle{.2}}
      \put(1,2){\circle{.15}}
      \put(1,2){\circle{.26}}
      \put(4,2){\circle{.15}}
      \put(4,2){\circle{.26}}
      \put(7,2){\circle{.15}}
      \put(7,2){\circle{.26}}
      \put(1,0){\circle{.15}}
      \put(1,0){\circle{.26}}
      \put(4,0){\circle{.15}}
      \put(4,0){\circle{.26}}
      \put(7,0){\circle{.15}}
      \put(7,0){\circle{.26}} 
      \put(-.3,-.3){$A_1$}
      \put(2.7,-.3){$A_2$}
      \put(5.7,-.3){$A_3$}

    \end{picture}
    \vspace*{5mm}
    \caption{The Ornithorynque origami $\Orni$: The label of an edge indicates
      to which square it is glued. Crossing one of the three bars from right to left
      leads one copy of $E$ higher (modulo 3), 
      i.e. from Square 1 to Square 6, from Square 5 to Square 4
      and from Square 3 to Square 2. $\Ornicov$ has the three singularities
      $\circ$, $\bullet$ and $\circledcirc$, each of order 2.}
    \label{figure-orni}
  \end{center}
\end{figure}
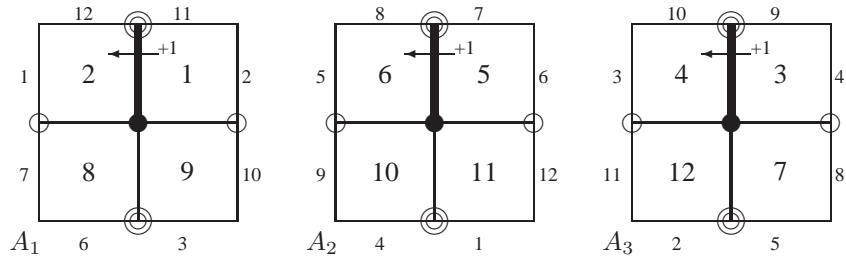

\begin{proposition}\label{jwe}
  Consider a translation covering $q_1:\Ornicov \to M_4$ and the following properties:
  \begin{enumerate}
  \item[A)]
    The ramification data of $q_1$ over the point $A_3$ is different from the ones over 
   $A_1$ and $A_2$.
  \item[B)]
   The ramification data of $q = \pi_2 \circ q_1$ over $A$ differ from the ramification data of $q$ over
    all other points.
  \item[C)]
    The Veech group $\Gamma(\Ornicov)$ of $\Ornicov$ is contained in $\Gamma(6)$.
  \end{enumerate}
  We then have:
  \begin{enumerate}
  \item[i)]
    Any affine homeomorphism $f$ of $\Ornicov$ descends to an affine homeomorphism $\bar{f}$
    of $\Orni$ which fixes the three singularities $\textbf{X}_1$, $\textbf{Y}_1$ and $\textbf{Z}_1$ and the three points $A_1$, $A_2$ and $A_3$ 
    pointwise.
  \item[ii)]
    The affine homeomorphisms of $\Ornicov$ act trivially on the lift of $H:=H_1^{(0)}(M_4,\mathbb{R})$ to $H_1(\Ornicov,\mathbb{R})$. 
    The Hodge bundle over the SL$(2, \R)$-orbit of $(\Ornicov,\omega)$
    thus contains non-trivial SL$(2, \R)$-invariant isotropic subbundles.
  \end{enumerate}
\end{proposition}

\begin{proof}
ii) directly follows from (\ref{orni-trivial-homology})  and 
i) in the same way as in Theorem~\ref{t.F-iso}.
To prove i) we first show that $f$ descends to $\Orni$.
This again follows from \cite[Lemma 3.13]{MatheusSchmithuesen} in the following way.
By the universality of the torus,
we know that $f$ descends via $q$ to some $g$ on $E$ which by B) fixes the
point $A$. Furthermore,
the Veech group of $\Orni$ is the full group $SL(2,\Z)$, see \cite[p. 473]{MY}. Thus there is some
affine homeomorphism $\bar{f}'$ of $\Orni$ with derivative $D(\bar{f}') = D(f)=D(g)$. 
It descends via $\pi_2$ to a homeomorphism of $E$ with the same derivative. The descend 
preserves two-division points and thus has to fix the point $A$, since this is the only
non-ramification point with respect to $\pi_2$ among the  two-division points. 
Thus the descend actually is equal to $g$ and 
$\bar{f}'$ is a lift of $g$. Now  \cite[Lemma 3.13]{MatheusSchmithuesen} tells us that $f$ descends via $q_1$ 
to some
affine homeomorphism $\bar{f}$ of $\Orni$.\\
We now show that $\bar{f}$ fixes the desired points: 
Since $g$ fixes $A$, we have that $\bar{f}$ preserves $\pi_2^{-1}(A) = \{A_1,A_2,A_3\}$. It follows then
from A) that $\bar{f}$ fixes the point $A_3$. Any
affine homeomorphism of $\Orni$ that fixes $A_3$ also fixes $A_1$ and $A_2$, see 
\cite[p.473]{MY}.
Furthermore, since the Veech group is contained in $\Gamma(2)$ by C), we obtain that  
$g$ fixes $\textbf{X}$, $\textbf{Y}$ and $\textbf{Z}$ pointwise. Thus we have that $\bar{f}$ fixes the singularities $\textbf{X}_1$, $\textbf{Y}_1$ and $\textbf{Z}_1$
pointwise.
\end{proof}

In the following we give an explicit construction of origamis $\Ornicov$ that
have the properties required  in Proposition~\ref{jwe}. We use for this the construction of what
we call {\em fake fibre product}, compare \cite[Section 1.2]{Nisbach},
which we define in the following.

\begin{definition}
Let $\pi_1:X_1 \to Y$ and $\pi_2: X_2 \to Y$ be two ramified coverings. Let $S$ be the union
of the set of ramification points of $\pi_1$ and that of $\pi_2$. Thus we obtain unramified
coverings $\pi_1: X_1^* \to Y^*$ and $\pi_2:X_2^* \to Y^*$, with $Y^* = Y\backslash S$, $X_1^* = X_1 \backslash \pi_1^{-1}(S)$
and   $X_2^* = X_2 \backslash \pi_2^{-1}(S)$. Let now $\rho_1$ and $\rho_2$ be the monodromy groups of these
unramified coverings, respectively. Hence they are maps from the fundamental group $\pi_1(Y^*)$ of $Y^*$
to the symmetric groups $S_{d_1}$, respectively $S_{d_2}$, where $d_1$ is the degree of $\pi_1$
and $d_2$ is the degree of $\pi_2$. Let $d = d_1\cdot d_2$ and identify $S_d$ with the symmetric group 
of the set $\{1,\ldots, d_1\} \times \{1, \ldots d_2\}$.  We then can consider the map 
\[\rho: \pi_1(Y^*) \to S_{d}, w \mapsto \rho(w) \mbox{ with }\rho(w)(a,b) = (\rho_1(w)(a),\rho_2(w)(b)).\]
Note that $\rho$ does not have to be a transitive action. If it is transitive, it defines a finite (connected) unramified
covering $q: Z^* \to Y^*$ for some punctured finite Riemann surface $Z^*$.
We then may extend $q$ to a ramified covering $q: Z \to Y$ between closed Riemann surfaces. We call $q$ the 
{\em fake fibre product}\footnote{ Observe that this is not the fibre product of $X_1$ and $X_2$
in the category of algebraic curves since that may have singularities.} of $\pi_1$ and $\pi_2$. 
If $\pi_1$ and $\pi_2$ are translation coverings, the translation
atlas on $Y^*$ can be lifted to $Z^*$ and $q$ becomes a translation covering.
\end{definition}

\begin{remark}Observe that the fake fibre product naturally comes with coverings $q_1: Z \to X_2$ and $q_2: Z \to X_1$
such that we have a commutative diagram
\[
\xymatrix{
  Z \ar^{q_2}[rr] \ar^{q_1}[d]&& X_1 \ar^{\pi_1}[d]\\
  X_2 \ar^{\pi_2}[rr] & & E
}
\]
\end{remark}

For the construction of the family of origamis $\Ornicov = \Ornicov(n)$, we now take as main ingredients which
will be defined next:
\begin{itemize}
\item 
  A covering $h:\Ornitwo \to \Orni$ which ramifies differently over $A_3$ than over $A_1$ and $A_2$
  and is unramified over all other points. We then work with $\tilde{\pi}_2 = \pi_2 \circ h: \Ornitwo \to E$. 
\item
  Coverings $\pi_1:Y = Y(n) \to E$ which ramify pairwise differently over $A$ and the two six-division
  points $P$ and $Q$ shown in Figure~\ref{rampoints}.
\item
  The fake fibre product $q:\Ornicov \to E$ of $\tilde{\pi}_2$ and $\pi_1$. We show that 
  we indeed obtain a connected surface. We denote the projection
  $\Ornicov \to \Ornitwo$ by $\tilde{q}_1$. 
\end{itemize}
The map $q_1 = h \circ \tilde{q}_1$ will then be the
map which we need for Proposition~\ref{jwe}, see also the diagram in Theorem~\ref{thm-orni}.\\

Let $\Ornitwo$ be the origami shown in Figure~\ref{Ornitwo} and $h:\Ornitwo \to \Orni$
the degree 2 covering to $\Orni$ obtained from mapping square $i$ on
$\Ornitwo$ to square $i \mbox{ mod } 12$  on $\Orni$.
You can directly read off the ramification data of $h$ from Figure~\ref{Ornitwo}. \\

  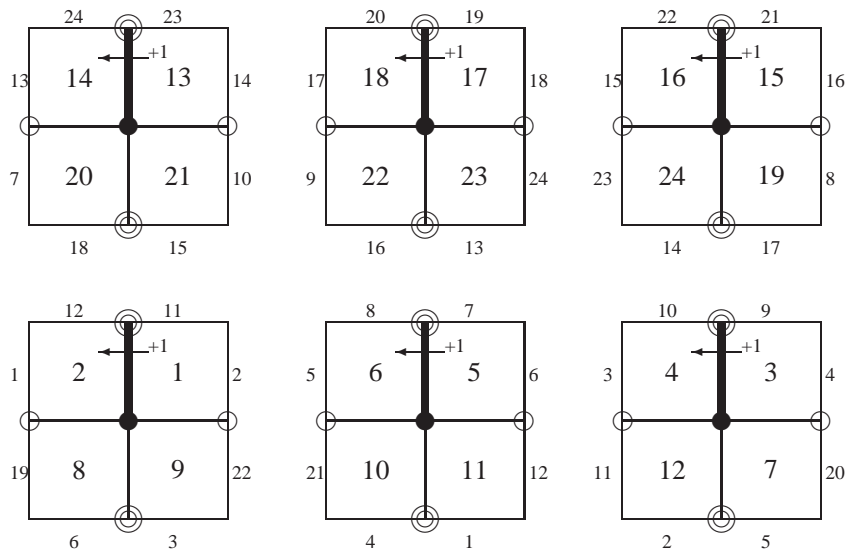
\begin{figure}[h]
    \setlength{\unitlength}{1.3cm}
    \begin{center}
      \begin{picture}(8,6)
        \put(0,0){\framebox(1,1){8}}
        \put(1,0){\framebox(1,1){9}}
        \put(3,0){\framebox(1,1){10}}
        \put(4,0){\framebox(1,1){11}}
        \put(6,0){\framebox(1,1){12}}
        \put(7,0){\framebox(1,1){7}}
        \put(0,1){\framebox(1,1){2}}
        \put(1,1){\framebox(1,1){1}}
        \put(3,1){\framebox(1,1){6}}
        \put(4,1){\framebox(1,1){5}}
        \put(6,1){\framebox(1,1){4}}
        \put(7,1){\framebox(1,1){3}}
        \put(0.4,-.3){\scs{6}}
        \put(1.4,-.3){\scs{3}}
        \put(3.4,-.3){\scs{4}}
        \put(4.4,-.3){\scs{1}}
        \put(6.4,-.3){\scs{2}}
        \put(7.4,-.3){\scs{5}}
        \put(0.35,2.05){\scs{12}}
        \put(1.35,2.05){\scs{11}}
        \put(3.4,2.05){\scs{8}}
        \put(4.4,2.05){\scs{7}}
        \put(6.35,2.05){\scs{10}}
        \put(7.4,2.05){\scs{9}}
        \put(-.2,.4){\scs{19}}
        \put(2.8,.4){\scs{21}}
        \put(5.7,.4){\scs{11}}
        \put(2.05,.4){\scs{22}}
        \put(5.05,.4){\scs{12}}
        \put(8.05,.4){\scs{20}}
        \put(-.2,1.4){\scs{1}}
        \put(2.8,1.4){\scs{5}}
        \put(5.8,1.4){\scs{3}}
        \put(2.05,1.4){\scs{2}}
        \put(5.05,1.4){\scs{6}}
        \put(8.05,1.4){\scs{4}}
        \put(1.2,1.7){\vector(-1,0){.5}}
        \put(1.2,1.7){\scs{+1}}
        \put(4.2,1.7){\vector(-1,0){.5}}
        \put(4.2,1.7){\scs{+1}}
        \put(7.2,1.7){\vector(-1,0){.5}}
        \put(7.2,1.7){\scs{+1}}

        \put(0,3){\framebox(1,1){20}}
        \put(1,3){\framebox(1,1){21}}
        \put(3,3){\framebox(1,1){22}}
        \put(4,3){\framebox(1,1){23}}
        \put(6,3){\framebox(1,1){24}}
        \put(7,3){\framebox(1,1){19}}
        \put(0,4){\framebox(1,1){14}}
        \put(1,4){\framebox(1,1){13}}
        \put(3,4){\framebox(1,1){18}}
        \put(4,4){\framebox(1,1){17}}
        \put(6,4){\framebox(1,1){16}}
        \put(7,4){\framebox(1,1){15}}
        \put(0.4,2.7){\scs{18}}
        \put(1.4,2.7){\scs{15}}
        \put(3.4,2.7){\scs{16}}
        \put(4.4,2.7){\scs{13}}
        \put(6.4,2.7){\scs{14}}
        \put(7.4,2.7){\scs{17}}
        \put(0.35,5.05){\scs{24}}
        \put(1.35,5.05){\scs{23}}
        \put(3.4,5.05){\scs{20}}
        \put(4.4,5.05){\scs{19}}
        \put(6.35,5.05){\scs{22}}
        \put(7.4,5.05){\scs{21}}
        \put(-.2,3.4){\scs{7}}
        \put(2.8,3.4){\scs{9}}
        \put(5.7,3.4){\scs{23}}
        \put(2.05,3.4){\scs{10}}
        \put(5.05,3.4){\scs{24}}
        \put(8.05,3.4){\scs{8}}
        \put(-.2,4.4){\scs{13}}
        \put(2.8,4.4){\scs{17}}
        \put(5.8,4.4){\scs{15}}
        \put(2.05,4.4){\scs{14}}
        \put(5.05,4.4){\scs{18}}
        \put(8.05,4.4){\scs{16}}
        \put(1.2,4.7){\vector(-1,0){.5}}
        \put(1.2,4.7){\scs{+1}}
        \put(4.2,4.7){\vector(-1,0){.5}}
        \put(4.2,4.7){\scs{+1}}
        \put(7.2,4.7){\vector(-1,0){.5}}
        \put(7.2,4.7){\scs{+1}}

        \linethickness{1mm}
        \put(1,1){\line(0,1){1}}
        \put(4,1){\line(0,1){1}}
        \put(7,1){\line(0,1){1}}
        \put(1,1){\circle*{.2}}
        \put(4,1){\circle*{.2}}
        \put(7,1){\circle*{.2}}
        \put(0,1){\circle{.2}}
        \put(3,1){\circle{.2}}
        \put(6,1){\circle{.2}}
        \put(2,1){\circle{.2}}
        \put(5,1){\circle{.2}}
        \put(8,1){\circle{.2}}
        \put(1,2){\circle{.15}}
        \put(1,2){\circle{.26}}
        \put(4,2){\circle{.15}}
        \put(4,2){\circle{.26}}
        \put(7,2){\circle{.15}}
        \put(7,2){\circle{.26}}
        \put(1,0){\circle{.15}}
        \put(1,0){\circle{.26}}
        \put(4,0){\circle{.15}}
        \put(4,0){\circle{.26}}
        \put(7,0){\circle{.15}}
        \put(7,0){\circle{.26}}
        
        \linethickness{1mm}
        \put(1,4){\line(0,1){1}}
        \put(4,4){\line(0,1){1}}
        \put(7,4){\line(0,1){1}}
        \put(1,4){\circle*{.2}}
        \put(4,4){\circle*{.2}}
        \put(7,4){\circle*{.2}}
        \put(0,4){\circle{.2}}
        \put(3,4){\circle{.2}}
        \put(6,4){\circle{.2}}
        \put(2,4){\circle{.2}}
        \put(5,4){\circle{.2}}
        \put(8,4){\circle{.2}}
        \put(1,5){\circle{.15}}
        \put(1,5){\circle{.26}}
        \put(4,5){\circle{.15}}
        \put(4,5){\circle{.26}}
        \put(7,5){\circle{.15}}
        \put(7,5){\circle{.26}}
        \put(1,3){\circle{.15}}
        \put(1,3){\circle{.26}}
        \put(4,3){\circle{.15}}
        \put(4,3){\circle{.26}}
        \put(7,3){\circle{.15}}
        \put(7,3){\circle{.26}}
      \end{picture}
    \end{center}
    \caption{Degree 2 covering $h: \Ornitwo \to \Orni$ of the Ornithorynque origami $\Orni$}\label{Ornitwo}
  \end{figure}

\begin{remark}\label{ramdata-pi2tilde}
The degree 2 covering $h:\Ornitwo \to \Orni$ is ramified over $A_1$ and $A_2$. 
All other points are unramified. Thus the map $\tilde{\pi}_2 = \pi_2 \circ h$
has the ramification data:
\begin{itemize}
\item
  over $A$: $(2,2,1,1)$
\item
  over $\textbf{X}$, $\textbf{Y}$ and $\textbf{Z}$: $(3,3)$
\end{itemize}
\end{remark}

For $n\in \N$ define the covering $\pi_1 =  \pi_1(n): Y = Y_n \to E$ as
follows: Take $n$ copies of the square shown in Figure~\ref{rampoints}.
Glue each edge to the opposite edge in the same copy, except for the 
edges $a$, $b$, $c$, $d$ and $e$.
Glue the edges labelled by $a$ according to  the permutation 
$\sigma_a = (1, \ldots, n)$, i.e. the upper edge of Square $31$ in Copy $i$ 
is glued to the lower edge of Square 1 in Copy $\sigma_a(i)$.
Similarly glue the edges labelled by $b$, $c$, $d$ and $e$   according to the permutation 
$\sigma_b = \sigma_c = \sigma_d = \sigma_e = (1,2)$. 
In the case of $e$ you glue the right edge of Square $2$ in Copy $i$ with
the left edge of Square 3 in Copy $\sigma_e(i)$. You can directly check from 
the construction that $Y_n$ is connected.

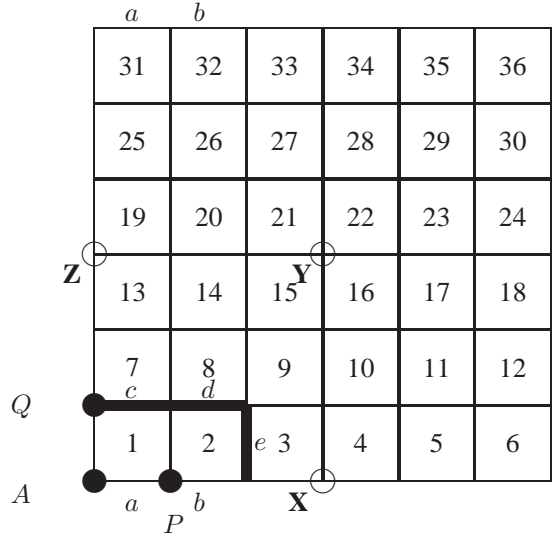
\begin{figure}[h]
  \setlength{\unitlength}{1cm}
  \begin{center}
    \begin{picture}(6,7)
      \put(0,0){\framebox(1,1){1}} 
      \put(0,1){\framebox(1,1){7}} 
      \put(0,2){\framebox(1,1){13}} 
      \put(0,3){\framebox(1,1){19}} 
      \put(0,4){\framebox(1,1){25}} 
      \put(0,5){\framebox(1,1){31}} 
      \put(1,0){\framebox(1,1){2}} 
      \put(1,1){\framebox(1,1){8}} 
      \put(1,2){\framebox(1,1){14}} 
      \put(1,3){\framebox(1,1){20}} 
      \put(1,4){\framebox(1,1){26}} 
      \put(1,5){\framebox(1,1){32}} 
      \put(2,0){\framebox(1,1){3}} 
      \put(2,1){\framebox(1,1){9}} 
      \put(2,2){\framebox(1,1){15}} 
      \put(2,3){\framebox(1,1){21}} 
      \put(2,4){\framebox(1,1){27}} 
      \put(2,5){\framebox(1,1){33}} 
      \put(3,0){\framebox(1,1){4}} 
      \put(3,1){\framebox(1,1){10}} 
      \put(3,2){\framebox(1,1){16}} 
      \put(3,3){\framebox(1,1){22}} 
      \put(3,4){\framebox(1,1){28}} 
      \put(3,5){\framebox(1,1){34}} 
      \put(4,0){\framebox(1,1){5}} 
      \put(4,1){\framebox(1,1){11}} 
      \put(4,2){\framebox(1,1){17}} 
      \put(4,3){\framebox(1,1){23}} 
      \put(4,4){\framebox(1,1){29}} 
      \put(4,5){\framebox(1,1){35}} 
      \put(5,0){\framebox(1,1){6}} 
      \put(5,1){\framebox(1,1){12}} 
      \put(5,2){\framebox(1,1){18}} 
      \put(5,3){\framebox(1,1){24}} 
      \put(5,4){\framebox(1,1){30}} 
      \put(5,5){\framebox(1,1){36}} 
      \put(0,0){\circle*{.3}}      
      \put(-1.1,-.3){$A$}
      \put(3,0){\circle{.3}}
      \put(2.55,-.4){$\textbf{X}$}
      \put(3,3){\circle{.3}}
      \put(2.6,2.63){$\textbf{Y}$}
      \put(0,3){\circle{.3}}
      \put(-.42,2.63){$\textbf{Z}$}
      \put(1,0){\circle*{.3}}
      \put(.9,-.7){$P$}
      \put(0,1){\circle*{.3}}
      \put(-1.1,.9){$Q$}
      \put(.4,-.4){$a$}
      \put(1.3,-.4){$b$}
      \put(.4,6.1){$a$}
      \put(1.3,6.1){$b$}
      \put(.4,1.1){$c$}
      \put(1.4,1.1){$d$}
      \put(2.1,.4){$e$}
     \linethickness{1.3mm}
     \put(0,1){\line(1,0){2}}
     \put(2,0){\line(0,1){1}} 
    \end{picture}\\[5mm]
  \end{center}
  \caption{Markings on the torus $E$}
  \label{rampoints}
\end{figure}

\begin{remark}\label{ramdata-pi1}
The covering $\pi_1 = \pi_1(n)$ has the following ramification data:
\begin{itemize}
\item
  over $A$: $(n)$
\item
  over $P$: $(1,n-1)$
\item
  over $Q$: $(2,\underbrace{1, \ldots, 1}_{n-2})$
\end{itemize}
and is unramified over all other points.
\end{remark}

\begin{theorem}\label{thm-orni}
The following construction gives an infinite sequence
of origamis $\Ornicov = \Ornicov(n)$ whose affine group act trivially on the lift of $H:=H_1^{(0)}(M_4,\mathbb{R})$ to $H_1(\Ornicov,\mathbb{R})$.\\
Let $n \geq 5$, $n$ odd, and let $\pi_1 = \pi_1(n)$ and $h$ be the coverings defined above.
Define $\tilde{\pi_2} = \pi_2 \circ h$. Take now the fake fibre product 
$q = \pi_1 \times \tilde{\pi}_2: \Ornicov \to E$ of $\pi_1$ and $\tilde{\pi}_2$. 
This comes by definitions
with two coverings $\tilde{q}_1: \Ornicov \to \Ornitwo$ and $q_2: \Ornicov \to Y$ 
such that $q =  \tilde{\pi}_2 \circ \tilde{q}_1 = \pi_1 \circ q_2$. We choose
on $E$ the translation structure coming from the identification
$E =  \C/(6\Z\oplus 6\Z i)$ and
denote its pullback to $\Ornicov$ via $q$ by $\omega$.
\[
\xymatrix{
  \Ornicov \ar^{q_2}[rr] \ar^{\tilde{q_1}}[d]&& Y \ar^{\pi_1}[d]\\
  \Ornitwo \ar^h[r] &\Orni \ar^{\pi_2}[r] & E
}
\]
\end{theorem}

Before giving the proof we study the ramification
data of the maps which we have just defined.
It follows from the definition of the fibre product that
the ramification data of the map $q$ over a point $p \in E$
are obtained from the ramification data of $\pi_1$ and $\tilde{\pi}_2$ over $p$, respectively.
More precisely we have that each pair of points $(p_1,p_2)$ in  $\pi_1^{-1}(p) \times \tilde{\pi}_2^{-1}(p)$,
where $p_1$ has ramification index $e$ with respect to $\pi_1$ and  $p_2$ has ramification index
$f$ with respect to $\tilde{\pi}_2$, produces $\frac{ef}{\lcmfs(e,f)} = \gcd(e,f)$
preimages of $p$ on $\Ornicov$ with ramification index $\lcmfs(e,f)$.\\ 
Similarly, we can read off the ramification
data of the map $\tilde{q}_1$ over a point $p_2$ in the fibre $ \tilde{\pi}_2^{-1}(p)$
from the ramification data of $\pi_1$ over $p$ and the ramification index $f$
of $\tilde{\pi}_2$ in $p_2$. More precisely each point $p_1$ in $\pi_1^{-1}(p)$ of ramification
index $e$ with respect to $\pi_1$ 
produces preimages with ramification index $\frac{\lcmfs(e,f)}{f}$ with
multiplicity $\frac{f\cdot ef}{f\lcmfs(e,f)} = \gcd(e,f)$.
In particular, if $\tilde{\pi}_2$ is unramified in $p_2$,
then the ramification data of $\tilde{q}_1$ over $p_2$ are equal to those of $\pi_1$ over $p$.\\
Let finally $\hat{A}_1$ and $\hat{A}_2$ be the preimage of $A_1$ and $A_2$ under $h$, respectively.
Let $\hat{A}_3^1$ and $\hat{A}_3^2$ be the two preimages of $A_3$.

\begin{remark}\label{ramdata}
From the previous considerations and Remark~\ref{ramdata-pi2tilde} and Remark~\ref{ramdata-pi1} 
we obtain the 
ramification data for $q$, $\tilde{q}_1$ and $q_1$.
  \begin{enumerate}
  \item[i)]
    The map $q$ has the following ramification data:
    \begin{itemize}
    \item over $A$: 
      \begin{tabular}[t]{l}
        $(2n,2n,n,n)$, if $n$ is odd;\\
        $(n,n,n,n,n,n)$, if $n$ is even.
      \end{tabular}
    \item over $\textbf{X}$, $\textbf{Y}$ and $\textbf{Z}$:
      $(\underbrace{3, \ldots, 3}_{2n})$
    \item over $P$:
      $(\underbrace{1, \ldots, 1}_6, \underbrace{n-1,\ldots, n-1}_6)$
    \item over $Q$:
      $(\underbrace{2,\ldots, 2}_6, \underbrace{1,\ldots,1}_{6n-12})$

    \end{itemize}
  \item[ii)]
    The map $\tilde{q}_1$ has the following ramification data:
    \begin{itemize}
    \item
      over $\hat{A}_1$ and $\hat{A}_2$: $(n)$, if $n$ is odd and $(\frac{n}{2},\frac{n}{2})$, if $n$ is even.
    \item
      over $\hat{A}_3^1$ and $\hat{A}_3^2$: $(n)$
    \item
      over each preimage of $\textbf{X}$, $\textbf{Y}$ or $\textbf{Z}$:$\underbrace{(1,\ldots, 1)}_{n}$
    \item
      over each preimage of  $P$: $(1,n-1)$
    \item
      over each preimage of $Q$: $(2,\underbrace{1, \ldots, 1}_{n-2})$.
    \end{itemize}
  \item[iii)]
    We finally obtain the ramification data of $q_1 =  h\circ \tilde{q}_1$:
    \begin{itemize}
    \item
      over $A_1$ and $A_2$: $(2n)$, if $n$ is odd and $(n,n)$, if $n$ is even.
    \item
      over $A_3$: $(n,n)$
    \item
      over the preimage of $\textbf{X}$, $\textbf{Y}$ or $\textbf{Z}$:$\underbrace{(1,\ldots, 1)}_{2n}$
    \item
      over each preimage of  $P$: $(1,1,n-1,n-1)$
    \item
    over each preimage of $Q$: $(2,2,\underbrace{1, \ldots, 1}_{2n-4})$.
    \end{itemize}
  \end{enumerate}
\end{remark}

\begin{proof}[Proof of Theorem~\ref{thm-orni}]
We first show that the fake fibre product $\Ornicov$ is connected
and then that $q_1 = h \circ \tilde{q}_1$ satisfies the assumptions 
of Proposition~\ref{jwe}.\\[2mm]
\textbf{Connectedness of $\Ornicov$}:
Let $\rho_1$ be the monodromy map of $\pi_1$ and $\rho_2$ that of
$\tilde{\pi}_2$. To make notation easier we remove the full set of  
six-division points from $E$ and call the resulting surface $E^{*}$.
Hence $\rho_1$ and $\rho_2$ are maps from the fundamental group $\pi_1(E^*)$
to $S_{n}$ and $S_6$, respectively. Also for the sake of simpler
notations we choose the base point of $\pi_1(E^*)$  in 
Square 7, see Figure~\ref{rampoints}. We label its preimages under $\pi_1$ 
on $Y = Y_n$ by the number of the corresponding sheet they lie in.
For the map $\tilde{\pi}_ 2$, observe that the six preimages of the base point lie 
in the squares labelled by $8$, $10$, $12$,
$20$, $22$, $24$, see Figure~\ref{Ornitwo}.
We label these preimages by the label $1$, $2$, $3$, \ldots, $6$, respectively.
We then have for the closed curves $x^6$ and $y^2x^6y^{-2}$ in $\pi_1(E^*)$ that
 $\rho_1(x^6)$ and  $\rho_1(y^2x^6y^{-2})$ are both trivial, since 
both paths do not cross one of the edges $a$, $b$, $c$, $d$ and $e$ (see Figure~\ref{rampoints}).
Furthermore, we have
\[ \rho_2(x^6) = (1,5,6)(2,3,4) \mbox{ and } \rho_2(y^2x^6y^{-2}) = (1,3,2)(4,6,5).\]
In particular
 $\rho_2(x^6)$ and $\rho_2(y^2x^6y^{-2})$ act transitively on $\{1,\ldots, 6\}$. Since $\rho_1$
acts transitively on $\{1,\ldots, n\}$ and furthermore trivially on the two elements
$x^6$ and $y^2x^6y^{-2}$, the action of the product $\rho_1 \times \rho_2$ is transitive
on $\{1, \ldots, n\}\times\{1, \ldots, 6\}$.
Hence $Y_n$ is connected.\\

\noindent
\textbf{Assumptions of Proposition~\ref{jwe}}: 
By Remark~\ref{ramdata}, $q_1$ satisfies the conditions A) and B) in Proposition~\ref{jwe}.
Consider an affine homeomorphism $f$ of $\Ornicov$. Let $g$ be its descend
to $E$ via $q$. By condition A), the map $g$ fixes the point $A$. By Remark~\ref{ramdata},
the ramification data of $q$ over $P$ and $Q$ are different from each other and different
from all other six-division points. Therefore $g$ also fixes $P$ and $Q$
and hence $D(f) = D(g)$ lie in $\Gamma(6)$. Thus we have $\Gamma(\Ornicov) \subseteq \Gamma(6)$
and also Condition C) from  Proposition~\ref{jwe} is fulfilled. 
 \end{proof}

We directly obtain from Remark~\ref{ramdata} the stratum of the surfaces $\Ornicov(n)$.

\begin{remark}
The origami $\Ornicov(n)$ with odd $n \geq 5$  lies in the stratum 
\[\mathcal{H}(2n-1,2n-1,n-1,n-1,\underbrace{n-2, \ldots, n-2}_{6}, \underbrace{2, \ldots, 2}_{6n},\underbrace{1, \ldots, 1}_{6})\]
and its genus is $12n-4$. 
\end{remark}

\acks{ The authors are thankful to Alex Eskin, Giovanni Forni, Pascal Hubert and Anton Zorich for useful discussions related to this note (including, of course, the questions at the origin of this text). The first author
is partially supported by the Landesstiftung Baden-W\"urttemberg (within the program {\em Junior\-pro\-fessu\-ren-Programm}).
The second author was partially supported by the French ANR grant ``GeoDyM'' (ANR-11-BS01-0004) and by the 
Balzan Research Project of J. Palis.}

\newpage

\begin{appendix}
\section{A picture of the origami \boldmath{$X$}}\label{a.A}

In Figure~\ref{all} below 
we show the full origami $X$ from Definition \ref{bigone}
which has 512 squares.\\[5mm]

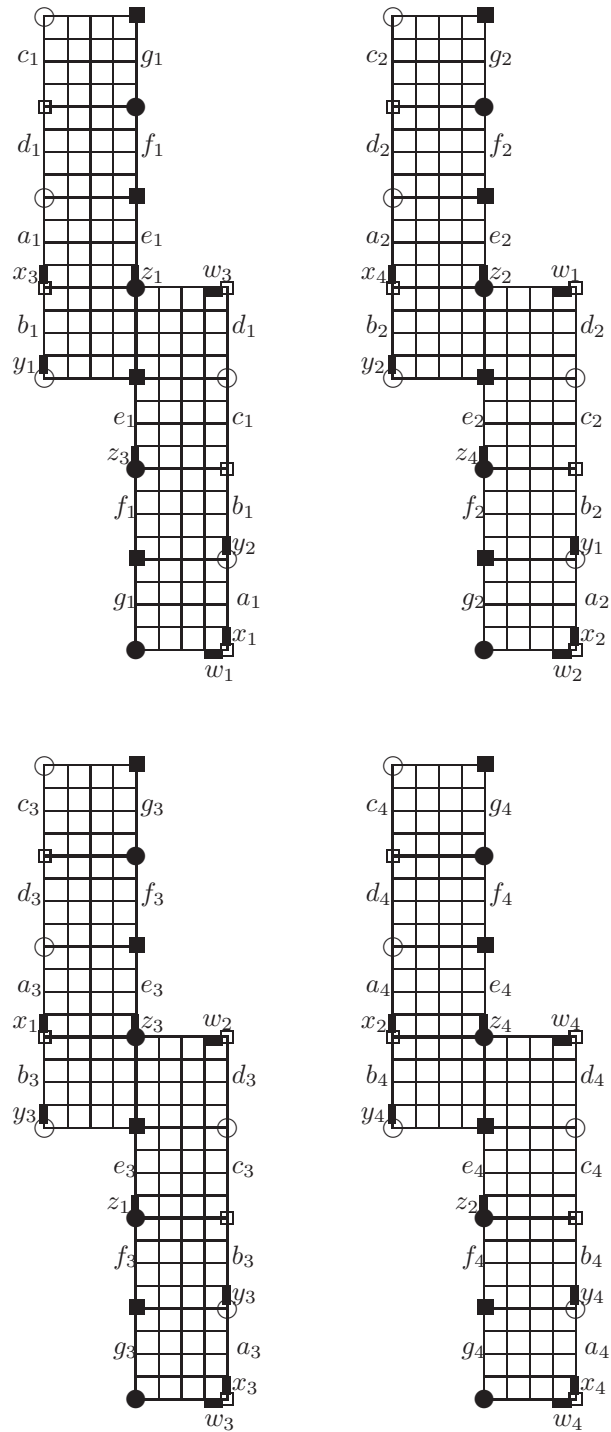
\begin{figure}[h!]
  \setlength{\unitlength}{1.2cm}
\begin{center}
  \begin{picture}(2,7)
    \put(1,0){\framebox(1,1){}}
    \put(1,1){\framebox(1,1){}}
    \put(1,2){\framebox(1,1){}}
    \put(1,3){\framebox(1,1){}}
    \put(0,3){\framebox(1,1){}}
    \put(0,4){\framebox(1,1){}}
    \put(0,5){\framebox(1,1){}}
    \put(0,6){\framebox(1,1){}}
    \put(1.94,-.06){\framebox(.12,.12){}}
    \put(1.94,1.94){\framebox(.12,.12){}}
    \put(1.94,3.94){\framebox(.12,.12){}}
    \put(-.06,3.94){\framebox(.12,.12){}}
    \put(-.06,5.94){\framebox(.12,.12){}}
    \put(2,1){\circle{.2}}
    \put(0,3){\circle{.2}}
    \put(0,5){\circle{.2}}
    \put(0,7){\circle{.2}}
    \put(2,3){\circle{.2}}
    \put(1,0){\circle*{.2}}
    \put(1,2){\circle*{.2}}
    \put(1,4){\circle*{.2}}
    \put(1,6){\circle*{.2}}    
    \put(0.93,0.93){\rule{2mm}{2mm}}
    \put(0.93,2.93){\rule{2mm}{2mm}}
    \put(0.93,4.93){\rule{2mm}{2mm}}
    \put(0.93,6.93){\rule{2mm}{2mm}}
    \put(0.25,3){\line(0,1){4} }
    \put(0.5,3){\line(0,1){4} }
    \put(0.75,3){\line(0,1){4} }
    \put(1.25,0){\line(0,1){4} }
    \put(1.5,0){\line(0,1){4} }
    \put(1.75,0){\line(0,1){4} }
    \put(0,6.25){\line(1,0){1} }
    \put(0,6.5){\line(1,0){1} }
    \put(0,6.75){\line(1,0){1} }
    \put(0,5.25){\line(1,0){1} }
    \put(0,5.5){\line(1,0){1} }
    \put(0,5.75){\line(1,0){1} }
    \put(0,4.25){\line(1,0){1} }
    \put(0,4.5){\line(1,0){1} }
    \put(0,4.75){\line(1,0){1} }
    \put(0,3.25){\line(1,0){2} }
    \put(0,3.5){\line(1,0){2} }    
    \put(0,3.75){\line(1,0){2} }
    \put(1,.25){\line(1,0){1} }
    \put(1,.5){\line(1,0){1} }
    \put(1,.75){\line(1,0){1} }
    \put(1,1.25){\line(1,0){1} }
    \put(1,1.5){\line(1,0){1} }
    \put(1,1.75){\line(1,0){1} }
    \put(1,2.25){\line(1,0){1} }
    \put(1,2.5){\line(1,0){1} }
    \put(1,2.75){\line(1,0){1} }
    \put(2.1,0.5){$a_1$}
    \put(-.3,4.5){$a_1$}
    \put(2.05,1.5){$b_1$}
    \put(-.3,3.5){$b_1$}
    \put(2.05,2.5){$c_1$}
    \put(-.3,6.5){$c_1$}
    \put(2.05,3.5){$d_1$}
    \put(-.3,5.5){$d_1$}
    \put(1.05,4.5){$e_1$}
    \put(.75,2.5){$e_1$}
    \put(1.05,5.5){$f_1$}
    \put(.75,1.5){$f_1$}
    \put(1.05,6.5){$g_1$}
    \put(.75,0.5){$g_1$}
    \put(2.05,.1){$x_1$}
    \put(-.35,4.1){$x_3$}
    \put(2.05,1.1){$y_2$}
    \put(-.35,3.1){$y_1$}
    \put(1.05,4.1){$z_1$}
    \put(.68,2.1){$z_3$}
    \put(1.75,-.3){$w_1$}
    \put(1.73,4.12){$w_3$}
    \linethickness{1mm}
    \put(1.99,0.05){\line(0,1){.2}}
    \put(-0.01,4.05){\line(0,1){.2}}
    \put(1.99,1.05){\line(0,1){.2}}
    \put(-0.01,3.05){\line(0,1){.2}}
    \put(.99,4.05){\line(0,1){.2}}
    \put(.99,2.05){\line(0,1){.2}}
    \put(1.75,-.05){\line(1,0){.2}}
    \put(1.75,3.95){\line(1,0){.2}}
    \linethickness{0.2mm}
  \end{picture}
  \hspace*{20mm}
  \begin{picture}(2,7)
    \put(1,0){\framebox(1,1){}}
    \put(1,1){\framebox(1,1){}}
    \put(1,2){\framebox(1,1){}}
    \put(1,3){\framebox(1,1){}}
    \put(0,3){\framebox(1,1){}}
    \put(0,4){\framebox(1,1){}}
    \put(0,5){\framebox(1,1){}}
    \put(0,6){\framebox(1,1){}}
    \put(1.94,-.06){\framebox(.12,.12){}}
    \put(1.94,1.94){\framebox(.12,.12){}}
    \put(1.94,3.94){\framebox(.12,.12){}}
    \put(-.06,3.94){\framebox(.12,.12){}}
    \put(-.06,5.94){\framebox(.12,.12){}}
    \put(2,1){\circle{.2}}
    \put(0,3){\circle{.2}}
    \put(0,5){\circle{.2}}
    \put(0,7){\circle{.2}}
    \put(2,3){\circle{.2}}
    \put(1,0){\circle*{.2}}
    \put(1,2){\circle*{.2}}
    \put(1,4){\circle*{.2}}
    \put(1,6){\circle*{.2}}    
    \put(0.93,0.93){\rule{2mm}{2mm}}
    \put(0.93,2.93){\rule{2mm}{2mm}}
    \put(0.93,4.93){\rule{2mm}{2mm}}
    \put(0.93,6.93){\rule{2mm}{2mm}}
    \put(0.25,3){\line(0,1){4} }
    \put(0.5,3){\line(0,1){4} }
    \put(0.75,3){\line(0,1){4} }
    \put(1.25,0){\line(0,1){4} }
    \put(1.5,0){\line(0,1){4} }
    \put(1.75,0){\line(0,1){4} }
    \put(0,6.25){\line(1,0){1} }
    \put(0,6.5){\line(1,0){1} }
    \put(0,6.75){\line(1,0){1} }
    \put(0,5.25){\line(1,0){1} }
    \put(0,5.5){\line(1,0){1} }
    \put(0,5.75){\line(1,0){1} }
    \put(0,4.25){\line(1,0){1} }
    \put(0,4.5){\line(1,0){1} }
    \put(0,4.75){\line(1,0){1} }
    \put(0,3.25){\line(1,0){2} }
    \put(0,3.5){\line(1,0){2} }    
    \put(0,3.75){\line(1,0){2} }
    \put(1,.25){\line(1,0){1} }
    \put(1,.5){\line(1,0){1} }
    \put(1,.75){\line(1,0){1} }
    \put(1,1.25){\line(1,0){1} }
    \put(1,1.5){\line(1,0){1} }
    \put(1,1.75){\line(1,0){1} }
    \put(1,2.25){\line(1,0){1} }
    \put(1,2.5){\line(1,0){1} }
    \put(1,2.75){\line(1,0){1} }
    \put(2.1,0.5){$a_2$}
    \put(-.3,4.5){$a_2$}
    \put(2.05,1.5){$b_2$}
    \put(-.3,3.5){$b_2$}
    \put(2.05,2.5){$c_2$}
    \put(-.3,6.5){$c_2$}
    \put(2.05,3.5){$d_2$}
    \put(-.3,5.5){$d_2$}
    \put(1.05,4.5){$e_2$}
    \put(.75,2.5){$e_2$}
    \put(1.05,5.5){$f_2$}
    \put(.75,1.5){$f_2$}
    \put(1.05,6.5){$g_2$}
    \put(.75,0.5){$g_2$}
    \put(2.05,.1){$x_2$}
    \put(-.35,4.1){$x_4$}
    \put(2.05,1.1){$y_1$}
    \put(-.35,3.1){$y_2$}
    \put(1.05,4.1){$z_2$}
    \put(.68,2.1){$z_4$}
    \put(1.75,-.3){$w_2$}
    \put(1.73,4.12){$w_1$}
    \linethickness{1mm}
    \put(1.99,0.05){\line(0,1){.2}}
    \put(-0.01,4.05){\line(0,1){.2}}
    \put(1.99,1.05){\line(0,1){.2}}
    \put(-0.01,3.05){\line(0,1){.2}}
    \put(.99,4.05){\line(0,1){.2}}
    \put(.99,2.05){\line(0,1){.2}}
    \put(1.75,-.05){\line(1,0){.2}}
    \put(1.75,3.95){\line(1,0){.2}}
    \linethickness{0.2mm}
  \end{picture}\\[15mm]
  \begin{picture}(2,7)
    \put(1,0){\framebox(1,1){}}
    \put(1,1){\framebox(1,1){}}
    \put(1,2){\framebox(1,1){}}
    \put(1,3){\framebox(1,1){}}
    \put(0,3){\framebox(1,1){}}
    \put(0,4){\framebox(1,1){}}
    \put(0,5){\framebox(1,1){}}
    \put(0,6){\framebox(1,1){}}
    \put(1.94,-.06){\framebox(.12,.12){}}
    \put(1.94,1.94){\framebox(.12,.12){}}
    \put(1.94,3.94){\framebox(.12,.12){}}
    \put(-.06,3.94){\framebox(.12,.12){}}
    \put(-.06,5.94){\framebox(.12,.12){}}
    \put(2,1){\circle{.2}}
    \put(0,3){\circle{.2}}
    \put(0,5){\circle{.2}}
    \put(0,7){\circle{.2}}
    \put(2,3){\circle{.2}}
    \put(1,0){\circle*{.2}}
    \put(1,2){\circle*{.2}}
    \put(1,4){\circle*{.2}}
    \put(1,6){\circle*{.2}}    
    \put(0.93,0.93){\rule{2mm}{2mm}}
    \put(0.93,2.93){\rule{2mm}{2mm}}
    \put(0.93,4.93){\rule{2mm}{2mm}}
    \put(0.93,6.93){\rule{2mm}{2mm}}
    \put(0.25,3){\line(0,1){4} }
    \put(0.5,3){\line(0,1){4} }
    \put(0.75,3){\line(0,1){4} }
    \put(1.25,0){\line(0,1){4} }
    \put(1.5,0){\line(0,1){4} }
    \put(1.75,0){\line(0,1){4} }
    \put(0,6.25){\line(1,0){1} }
    \put(0,6.5){\line(1,0){1} }
    \put(0,6.75){\line(1,0){1} }
    \put(0,5.25){\line(1,0){1} }
    \put(0,5.5){\line(1,0){1} }
    \put(0,5.75){\line(1,0){1} }
    \put(0,4.25){\line(1,0){1} }
    \put(0,4.5){\line(1,0){1} }
    \put(0,4.75){\line(1,0){1} }
    \put(0,3.25){\line(1,0){2} }
    \put(0,3.5){\line(1,0){2} }    
    \put(0,3.75){\line(1,0){2} }
    \put(1,.25){\line(1,0){1} }
    \put(1,.5){\line(1,0){1} }
    \put(1,.75){\line(1,0){1} }
    \put(1,1.25){\line(1,0){1} }
    \put(1,1.5){\line(1,0){1} }
    \put(1,1.75){\line(1,0){1} }
    \put(1,2.25){\line(1,0){1} }
    \put(1,2.5){\line(1,0){1} }
    \put(1,2.75){\line(1,0){1} }
   \put(1,2.75){\line(1,0){1} }
    \put(2.1,0.5){$a_3$}
    \put(-.3,4.5){$a_3$}
    \put(2.05,1.5){$b_3$}
    \put(-.3,3.5){$b_3$}
    \put(2.05,2.5){$c_3$}
    \put(-.3,6.5){$c_3$}
    \put(2.05,3.5){$d_3$}
    \put(-.3,5.5){$d_3$}
    \put(1.05,4.5){$e_3$}
    \put(.75,2.5){$e_3$}
    \put(1.05,5.5){$f_3$}
    \put(.75,1.5){$f_3$}
    \put(1.05,6.5){$g_3$}
    \put(.75,0.5){$g_3$}
    \put(2.05,.1){$x_3$}
    \put(-.35,4.1){$x_1$}
    \put(2.05,1.1){$y_3$}
    \put(-.35,3.1){$y_3$}
    \put(1.05,4.1){$z_3$}
    \put(.68,2.1){$z_1$}
    \put(1.75,-.3){$w_3$}
    \put(1.73,4.12){$w_2$}
    \linethickness{1mm}
    \put(1.99,0.05){\line(0,1){.2}}
    \put(-0.01,4.05){\line(0,1){.2}}
    \put(1.99,1.05){\line(0,1){.2}}
    \put(-0.01,3.05){\line(0,1){.2}}
    \put(.99,4.05){\line(0,1){.2}}
    \put(.99,2.05){\line(0,1){.2}}
    \put(1.75,-.05){\line(1,0){.2}}
    \put(1.75,3.95){\line(1,0){.2}}
    \linethickness{0.2mm}
  \end{picture}
  \hspace*{20mm}
  \begin{picture}(2,7)
    \put(1,0){\framebox(1,1){}}
    \put(1,1){\framebox(1,1){}}
    \put(1,2){\framebox(1,1){}}
    \put(1,3){\framebox(1,1){}}
    \put(0,3){\framebox(1,1){}}
    \put(0,4){\framebox(1,1){}}
    \put(0,5){\framebox(1,1){}}
    \put(0,6){\framebox(1,1){}}
    \put(1.94,-.06){\framebox(.12,.12){}}
    \put(1.94,1.94){\framebox(.12,.12){}}
    \put(1.94,3.94){\framebox(.12,.12){}}
    \put(-.06,3.94){\framebox(.12,.12){}}
    \put(-.06,5.94){\framebox(.12,.12){}}
    \put(2,1){\circle{.2}}
    \put(0,3){\circle{.2}}
    \put(0,5){\circle{.2}}
    \put(0,7){\circle{.2}}
    \put(2,3){\circle{.2}}
    \put(1,0){\circle*{.2}}
    \put(1,2){\circle*{.2}}
    \put(1,4){\circle*{.2}}
    \put(1,6){\circle*{.2}}    
    \put(0.93,0.93){\rule{2mm}{2mm}}
    \put(0.93,2.93){\rule{2mm}{2mm}}
    \put(0.93,4.93){\rule{2mm}{2mm}}
    \put(0.93,6.93){\rule{2mm}{2mm}}
    \put(0.25,3){\line(0,1){4} }
    \put(0.5,3){\line(0,1){4} }
    \put(0.75,3){\line(0,1){4} }
    \put(1.25,0){\line(0,1){4} }
    \put(1.5,0){\line(0,1){4} }
    \put(1.75,0){\line(0,1){4} }
    \put(0,6.25){\line(1,0){1} }
    \put(0,6.5){\line(1,0){1} }
    \put(0,6.75){\line(1,0){1} }
    \put(0,5.25){\line(1,0){1} }
    \put(0,5.5){\line(1,0){1} }
    \put(0,5.75){\line(1,0){1} }
    \put(0,4.25){\line(1,0){1} }
    \put(0,4.5){\line(1,0){1} }
    \put(0,4.75){\line(1,0){1} }
    \put(0,3.25){\line(1,0){2} }
    \put(0,3.5){\line(1,0){2} }    
    \put(0,3.75){\line(1,0){2} }
    \put(1,.25){\line(1,0){1} }
    \put(1,.5){\line(1,0){1} }
    \put(1,.75){\line(1,0){1} }
    \put(1,1.25){\line(1,0){1} }
    \put(1,1.5){\line(1,0){1} }
    \put(1,1.75){\line(1,0){1} }
    \put(1,2.25){\line(1,0){1} }
    \put(1,2.5){\line(1,0){1} }
    \put(1,2.75){\line(1,0){1} }
    \put(2.1,0.5){$a_4$}
    \put(-.3,4.5){$a_4$}
    \put(2.05,1.5){$b_4$}
    \put(-.3,3.5){$b_4$}
    \put(2.05,2.5){$c_4$}
    \put(-.3,6.5){$c_4$}
    \put(2.05,3.5){$d_4$}
    \put(-.3,5.5){$d_4$}
    \put(1.05,4.5){$e_4$}
    \put(.75,2.5){$e_4$}
    \put(1.05,5.5){$f_4$}
    \put(.75,1.5){$f_4$}
    \put(1.05,6.5){$g_4$}
    \put(.75,0.5){$g_4$}
    \put(2.05,.1){$x_4$}
    \put(-.35,4.1){$x_2$}
    \put(2.05,1.1){$y_4$}
    \put(-.35,3.1){$y_4$}
    \put(1.05,4.1){$z_4$}
    \put(.68,2.1){$z_2$}
    \put(1.75,-.3){$w_4$}
    \put(1.73,4.12){$w_4$}
    \linethickness{1mm}
    \put(1.99,0.05){\line(0,1){.2}}
    \put(-0.01,4.05){\line(0,1){.2}}
    \put(1.99,1.05){\line(0,1){.2}}
    \put(-0.01,3.05){\line(0,1){.2}}
    \put(.99,4.05){\line(0,1){.2}}
    \put(.99,2.05){\line(0,1){.2}}
    \put(1.75,-.05){\line(1,0){.2}}
    \put(1.75,3.95){\line(1,0){.2}}
    \linethickness{0.2mm}
  \end{picture}
\end{center}
  \caption{The origami $X$ from Definition~\ref{bigone}}
  \label{all}
\end{figure}

\end{appendix}

\newpage

\end{document}